\documentclass[a4paper]{amsart}

\usepackage[utf8]{inputenc}
\usepackage[T1]{fontenc}
\usepackage[english]{babel}

\usepackage{amsfonts}
\usepackage{amsmath}
\usepackage{amsrefs}
\usepackage{amssymb}
\usepackage{amstext}
\usepackage{amsthm}

\usepackage{geometry}
\usepackage{fullpage}

\usepackage[shortlabels]{enumitem}
\usepackage{graphicx}
\usepackage{hyperref}
\usepackage{listings}
\usepackage{mathrsfs}
\usepackage{mathtools}
\usepackage{multicol}
\usepackage{shuffle}

\usepackage{tikz}
\usepackage{tikz-cd}
\usetikzlibrary{arrows,backgrounds,chains,decorations.pathreplacing,patterns,shapes}

\newcommand{\area}{\mathsf{area}}
\newcommand{\dinv}{\mathsf{dinv}}
\newcommand{\inv}{\mathsf{inv}}

\newcommand{\maj}{\mathsf{maj}}
\newcommand{\ides}{\mathsf{ides}}

\newcommand{\Rise}{\mathsf{Rise}}

\newcommand{\DRise}{\mathsf{DRise}}

\newcommand{\D}{\mathsf{D}} 

\newcommand{\PLD}{\mathsf{PLD}} 

\newcommand{\gendelta}{\mathsf{PLD}_{\underline x,q,t}(m,n)^{\ast k}} 
\newcommand{\OP}{\mathsf{OP}} 
\newcommand{\OPi}{\mathsf{OPi} } 
\newcommand{\OPd}{\mathsf{OPd} } 
\newcommand{\SOP}{\mathsf{SOP}} 

\newcommand{\qbinom}[2]{\genfrac{[}{]}{0pt}{}{#1}{#2}}

\makeatletter

\pgfkeys{
	/tikz/sharp angle/.code={%
		\pgfsetarrowoptions{sharp >}{#1}%
		\pgfsetarrowoptions{sharp <}{-#1}%
	},
	/tikz/sharp > angle/.code={%
		\pgfsetarrowoptions{sharp >}{#1}%
	},
	/tikz/sharp < angle/.code={%
		\pgfsetarrowoptions{sharp <}{#1}%
	},
	/tikz/sharp protrude/.code=\csname if#1\endcsname\qrr@tikz@sharp@z@-0.05\p@\else\qrr@tikz@sharp@z@\z@\fi,
	/tikz/sharp protrude/.default=true
}

\newdimen\qrr@tikz@sharp@z@
\qrr@tikz@sharp@z@\z@
\pgfarrowsdeclare{sharp >}{sharp >}{%
	\edef\pgf@marshal{\noexpand\pgfutil@in@{and}{\pgfgetarrowoptions{sharp >}}}%
	\pgf@marshal
	\ifpgfutil@in@
	\edef\pgf@tempa{\pgfgetarrowoptions{sharp >}}
	\expandafter\qrr@tikz@sharp@parse\pgf@tempa\@qrr@tikz@sharp@parse
	\else
	\qrr@tikz@sharp@parse\pgfgetarrowoptions{sharp >}and-\pgfgetarrowoptions{sharp >}\@qrr@tikz@sharp@parse
	\fi
	\pgfmathparse{max(\pgf@tempa,\pgf@tempb,0)}%
	\let\qrr@tikz@sharp@max\pgfmathresult
	\pgfmathsetlength\pgf@xa{.5*\pgflinewidth * tan(\qrr@tikz@sharp@max)}%
	\pgfarrowsleftextend{+\pgf@xa}%
	\pgfarrowsrightextend{+\pgf@xa}%
}{%
	\edef\pgf@marshal{\noexpand\pgfutil@in@{and}{\pgfgetarrowoptions{sharp >}}}%
	\pgf@marshal
	\ifpgfutil@in@
	\edef\pgf@tempa{\pgfgetarrowoptions{sharp >}}
	\expandafter\qrr@tikz@sharp@parse\pgf@tempa\@qrr@tikz@sharp@parse
	\else
	\qrr@tikz@sharp@parse\pgfgetarrowoptions{sharp >}and-\pgfgetarrowoptions{sharp >}\@qrr@tikz@sharp@parse
	\fi
	\pgfmathsetlength\pgf@ya{.5*\pgflinewidth * tan(max(\pgf@tempa,\pgf@tempb,0))}%
	\pgfmathsetlength\pgf@xa{-.5*\pgflinewidth * tan(\pgf@tempa)}%
	\pgfmathsetlength\pgf@xb{-.5*\pgflinewidth * tan(\pgf@tempb)}%
	\advance\pgf@xa\pgf@ya
	\advance\pgf@xb\pgf@ya
	\ifdim\pgf@xa>\pgf@xb
	\pgftransformyscale{-1}%
	\pgf@xc\pgf@xb
	\pgf@xb\pgf@xa
	\pgf@xa\pgf@xc
	\fi
	\pgfpathmoveto{\pgfqpoint{\qrr@tikz@sharp@z@}{.5\pgflinewidth}}%
	\pgfpathlineto{\pgfqpoint{\pgf@xa}{.5\pgflinewidth}}%
	\pgfpathlineto{\pgfqpoint{\pgf@ya}{+0pt}}%
	\pgfpathlineto{\pgfqpoint{\pgf@xb}{-.5\pgflinewidth}}%
	\pgfpathlineto{\pgfqpoint{\qrr@tikz@sharp@z@}{-.5\pgflinewidth}}%
	\pgfusepathqfill
}
\pgfarrowsdeclare{sharp <}{sharp <}{%
	\edef\pgf@marshal{\noexpand\pgfutil@in@{and}{\pgfgetarrowoptions{sharp <}}}%
	\pgf@marshal
	\ifpgfutil@in@
	\edef\pgf@tempa{\pgfgetarrowoptions{sharp <}}
	\expandafter\qrr@tikz@sharp@parse\pgf@tempa\@qrr@tikz@sharp@parse
	\else
	\expandafter\qrr@tikz@sharp@parse\pgfgetarrowoptions{sharp <}and-\pgfgetarrowoptions{sharp <}\@qrr@tikz@sharp@parse
	\fi
	\pgfmathparse{max(\pgf@tempa,\pgf@tempb,0)}%
	\let\qrr@tikz@sharp@max\pgfmathresult
	\pgfmathsetlength\pgf@xa{.5*\pgflinewidth * tan(\qrr@tikz@sharp@max)}%
	\pgfarrowsleftextend{+\pgf@xa}%
	\pgfarrowsrightextend{+\pgf@xa}%
}{%
	\edef\pgf@marshal{\noexpand\pgfutil@in@{and}{\pgfgetarrowoptions{sharp <}}}%
	\pgf@marshal
	\ifpgfutil@in@
	\edef\pgf@tempa{\pgfgetarrowoptions{sharp <}}
	\expandafter\qrr@tikz@sharp@parse\pgf@tempa\@qrr@tikz@sharp@parse
	\else
	\expandafter\qrr@tikz@sharp@parse\pgfgetarrowoptions{sharp <}and-\pgfgetarrowoptions{sharp <}\@qrr@tikz@sharp@parse
	\fi
	\pgfmathsetlength\pgf@ya{.5*\pgflinewidth * tan(max(\pgf@tempa,\pgf@tempb,0))}%
	\pgfmathsetlength\pgf@xa{-.5*\pgflinewidth * tan(\pgf@tempa)}%
	\pgfmathsetlength\pgf@xb{-.5*\pgflinewidth * tan(\pgf@tempb)}%
	\advance\pgf@xa\pgf@ya
	\advance\pgf@xb\pgf@ya
	\ifdim\pgf@xa>\pgf@xb
	\pgftransformyscale{-1}%
	\pgf@xc\pgf@xb
	\pgf@xb\pgf@xa
	\pgf@xa\pgf@xc
	\fi
	\pgfpathmoveto{\pgfqpoint{\qrr@tikz@sharp@z@}{.5\pgflinewidth}}%
	\pgfpathlineto{\pgfqpoint{\pgf@xa}{.5\pgflinewidth}}%
	\pgfpathlineto{\pgfqpoint{\pgf@ya}{+0pt}}%
	\pgfpathlineto{\pgfqpoint{\pgf@xb}{-.5\pgflinewidth}}%
	\pgfpathlineto{\pgfqpoint{\qrr@tikz@sharp@z@}{-.5\pgflinewidth}}%
	\pgfusepathqfill
}
\def\qrr@tikz@sharp@parse#1and#2\@qrr@tikz@sharp@parse{\def\pgf@tempa{#1}\def\pgf@tempb{#2}}

\makeatother

\newcommand\multiset[2]%
{\mathchoice{\left(\kern-0.4em{\binom{#1}{#2}}\kern-0.4em\right)}
	{\bigl(\kern-0.2em{\binom{#1}{#2}}\kern-0.2em\bigr)}
	{\bigl(\kern-0.2em{\binom{#1}{#2}}\kern-0.2em\bigr)}
	{\bigl(\kern-0.2em{\binom{#1}{#2}}\kern-0.2em\bigr)}}

\let\existstemp\exists \renewcommand*{\exists}{\mathop \existstemp}
\let\foralltemp\forall \renewcommand*{\forall}{\mathop \foralltemp}

\def\quotient#1#2{\raise1ex\hbox{$#1$}\Big/\lower1ex\hbox{$#2$}}

\newcommand{\<}{\langle}
\renewcommand{\>}{\rangle}

\newtheorem{theorem}{Theorem}[section]
\newtheorem{lemma}[theorem]{Lemma}
\newtheorem{proposition}[theorem]{Proposition}
\newtheorem{corollary}[theorem]{Corollary}
\newtheorem{conjecture}[theorem]{Conjecture}

\theoremstyle{definition}
\newtheorem{definition}[theorem]{Definition}

\newtheorem{example}[theorem]{Example}

\theoremstyle{remark}
\newtheorem{remark}[theorem]{Remark}

\title{The generalized Delta conjecture at $t=0$}

\author{Michele D'Adderio}
\address{Universit\'e Libre de Bruxelles (ULB)\\D\'epartement de Math\'ematique\\ Boulevard du Triomphe, B-1050 Bruxelles\\ Belgium}\email{mdadderi@ulb.ac.be}

\author{Alessandro Iraci}
\address{Universit\'a di Pisa and Universit\'e Libre de Bruxelles (ULB)\\Dipartimento di Matematica\\ Largo Bruno Pontecorvo 5, 56127 Pisa\\ Italia}\email{iraci@student.dm.unipi.it}

\author{Anna Vanden Wyngaerd}
\address{Universit\'e Libre de Bruxelles (ULB)\\D\'epartement de Math\'ematique\\ Boulevard du Triomphe, B-1050 Bruxelles\\ Belgium}\email{anvdwyng@ulb.ac.be}

\begin{document}
	
\begin{abstract}
We prove the cases $q=0$ and $t=0$ of the \emph{generalized Delta conjecture} of Haglund, Remmel and Wilson \cite{Haglund-Remmel-Wilson-2015} involving the symmetric function $\Delta_{h_m}\Delta_{e_{n-k-1}}'e_n$. Our theorem generalizes recent results by Garsia, Haglund, Remmel and Yoo \cite{Garsia-Haglund-Remmel-Yoo-2017}. This proves also the case $q=0$ of our recent \emph{generalized Delta square conjecture} \cite{DAdderio-Iraci-VandenWyngaerd-Delta-Square}.
\end{abstract}
	
\maketitle
\tableofcontents

\section{Introduction}
In \cite{Haglund-Remmel-Wilson-2015} Haglund Remmel and Wilson stated the so called \emph{Delta conjecture}, which can be written as
\[\Delta_{e_{n-k-1}}'e_n=\sum_{P\in \PLD(0,n)}q^{\dinv(P)}t^{\area(P)}x^P, \]
where on the left hand side we have one of the Delta operators introduced in \cite{Bergeron-Garsia-Haiman-Tesler-Positivity-1999} applied to the symmetric function $e_n$, and on the right hand side we have a combinatorial formula given in terms of labelled Dyck paths (see Sections~\ref{sec: combinat} and \ref{sec: symfun} for precise definitions). This formula generalizes the so called \emph{Shuffle conjecture} in \cite{HHLRU-2005} (which is the case $k=0$), recently proved in \cite{Carlsson-Mellit-ShuffleConj-2015} by Carlsson and Mellit.

The Delta conjecture already attracted quite a bit of interest, and several special cases have been proved (e.g. see \cite{DAdderio-Iraci-VandenWyngaerd-GenDeltaSchroeder} and references therein). In particular, the special cases $q=0$ and $t=0$ have been recently proved in \cite{Garsia-Haglund-Remmel-Yoo-2017} by Garsia, Haglund, Remmel and Yoo. To this day, the full conjecture remains widely open.

In the same \cite{Haglund-Remmel-Wilson-2015}, the authors formulated a more general conjecture, that we call \emph{generalized Delta conjecture}, and that can be stated as
\[\Delta_{h_m}\Delta_{e_{n-k-1}}'e_n=\sum_{P\in \PLD(m,n)}q^{\dinv{P}}t^{\area(P)}x^P, \]
where now on the left hand side we act with another Delta operator, while on the right hand side we sum over \emph{partially labelled Dyck paths} (again, see Sections~\ref{sec: combinat} and \ref{sec: symfun} for precise definitions). The Delta conjecture is simply the case $m=0$ of this one. 

In \cite{DAdderio-Iraci-VandenWyngaerd-GenDeltaSchroeder} we proved the so called \emph{Schr\"{o}der case}, i.e. the case $\<\cdot , e_{n-d}h_d\>$, of the generalized Delta conjecture.

The main result of this paper is to prove the special cases $q=0$ and $t=0$ of the generalized Delta conjecture:
\begin{theorem}
	For $m,n,k\in \mathbb{N}$, $m\geq 0$ and $n>k\geq 0$, we have both
	\[
	\left.\Delta_{h_{m}} \Delta'_{e_{n-k-1}} e_{n}\right|_{t=0} = \mathop{\sum_{P\in \PLD(m,n)}}_{\area(P)=0}q^{\dinv(P)} x^P
	\]
	and
	\[ \left.\Delta_{h_{m}} \Delta'_{e_{n-k-1}} e_{n}\right|_{q=0} = \mathop{\sum_{P\in \PLD(m,n)}}_{\dinv(P)=0} t^{\area(P)}x^P.\]
\end{theorem}
Notice that this result generalizes the main result in \cite{Garsia-Haglund-Remmel-Yoo-2017}.
\begin{remark}
	We would like to emphasize that our proof is independent of the results in \cite{Garsia-Haglund-Remmel-Yoo-2017}, hence providing a further new proof of the Delta conjecture at $q=0$ or $t=0$, after the alternative proofs in \cite{Haglund-Rhoades-Shimozono-arxiv} and in \cite{DAdderio-Iraci-VandenWyngaerd-Delta-Square}.
	
	It should also be noticed that our proof has the peculiar property that it does not specialize to the case $m=0$: for our argument to go through, the full generalized Delta conjecture at $t=0$ or $q=0$ is needed.
\end{remark}
Finally, in \cite{DAdderio-Iraci-VandenWyngaerd-Delta-Square} we proposed what we called a \emph{generalized Delta square conjecture}, which extends the \emph{square conjecture} of Loehr and Warrington \cite{Loehr-Warrington-square-2007}, recently proved by Sergel \cite{Leven-2016}, and it reduces to the generalized Delta conjecture when $q=0$: hence we proved also the case $q=0$ of this newer conjecture.

\medskip

The paper is organized in the following way. In Section~\ref{sec: combinat} we recall the combinatorial definitions needed for stating the generalized Delta conjecture. In Section~\ref{sec: symfun} we introduce some notation and we prove the identities of symmetric function theory needed in the following sections. In Section~4 we recall some results from Wilson \cite{Wilson-Equidistribution} about ordered set partitions, and determine some properties that we will need in Section~5 to show the equivalence between the cases $q=0$ and $t=0$ of the generalized Delta conjecture, and to reduce their proof to an identity involving ordered set partitions. In Section~6 we establish the main recursive combinatorial steps, that we will use in Section~7 to complete the proof of our main results. We conclude the article with some open problems

\section{The generalized Delta conjecture}\label{sec: combinat}

\emph{We refer to Section~\ref{sec: symfun} for notations and definitions concerning symmetric functions.}

\medskip

In \cite{Haglund-Remmel-Wilson-2015}, the authors conjectured a combinatorial interpretation for the symmetric function \[ \Delta_{h_m}\Delta'_{e_{n-k-1}}e_n \] in terms of partially labelled decorated Dyck paths, known as the \emph{generalized Delta conjecture} because it reduces to the Delta conjecture when $m=0$. We give the necessary definitions.

\begin{definition}
	A \emph{Dyck path} of size $n$ is a lattice path going from $(0,0)$ to $(n,n)$, using only north and east unit steps and staying weakly above the line $x=y$ (also called the \emph{main diagonal}). The set of Dyck paths of size $n$ will be denoted by $\mathsf{D}(n)$. A \emph{partially labelled Dyck path} is a Dyck path whose vertical steps are labelled with (not necessarily distinct) non-negative integers such that the labels appearing in each column are strictly increasing from bottom to top, and $0$ does not appear in the first column. The set of partially labelled Dyck paths with $m$ zero labels and $n$ nonzero labels is denoted by $\PLD(m,n)$.  
\end{definition}

Partially labelled Dyck paths differ from labelled Dyck paths only in that $0$ is allowed as a label in the former and not in the latter. 

\begin{definition}\label{def: monomial path}
	We define for each $D\in \PLD(m,n)$ a monomial in the variables $x_1,x_2,\dots$: we set \[ x^P \coloneqq \prod_{i=1}^{n+m} x_{l_i(P)} \] where $l_i(D)$ is the label of the $i$-th vertical step of $D$ (the first being at the bottom), where we conventionally set $x_0 = 1$. The fact that $x_0$ does not appear in the monomial explains the word \emph{partially}.
\end{definition}

\begin{definition}
	Let $D$ be a (partially labelled) Dyck path of size $n+m$. We define its \emph{area word} to be the string of integers $a(D) = a_1(D) \cdots a_{n+m}(D)$ where $a_i(D)$ is the number of whole squares in the $i$-th row (counting from the bottom) between the path and the main diagonal.
\end{definition}

\begin{definition} \label{def: rise}
	The \emph{rises} of a Dyck path $D$ are the indices \[ \Rise(D) \coloneqq \{2\leq i \leq n+m\mid a_{i}(D)>a_{i-1}(D)\},\] or the vertical steps that are directly preceded by another vertical step. Taking a subset $\DRise(D)\subseteq \Rise (D)$ and decorating the corresponding vertical steps with a $\ast$, we obtain a \emph{decorated Dyck path}, and we will refer to these vertical steps as \emph{decorated rises}. 
\end{definition}

\begin{definition}
	Given a partially labelled Dyck path, we call \emph{zero valleys} its vertical steps with label $0$ (which are necessarily preceded by an horizontal step, that is why we call them valleys).
\end{definition}

The set of partially labelled decorated Dyck paths with $m$ zero labels, $n$ nonzero labels and $k$ decorated rises is denoted by $\PLD(m,n)^{\ast k}$. See Figure~\ref{fig:pldExample1} for an example. 

\begin{figure*}[!ht]
	\centering
	\begin{tikzpicture}[scale = .8]
	
	\draw[step=1.0, gray!60, thin] (0,0) grid (8,8);
	
	\draw[gray!60, thin] (0,0) -- (8,8);
	
	\draw[blue!60, line width=2pt] (0,0) -- (0,1) -- (0,2) -- (1,2) -- (2,2) -- (2,3) -- (2,4) -- (2,5) -- (3,5) -- (4,5) -- (4,6) -- (4,7) -- (4,8) -- (5,8) -- (6,8) -- (7,8) -- (8,8);
	
	\draw (0.5,0.5) circle (0.4 cm) node {$1$};
	\draw (0.5,1.5) circle (0.4 cm) node {$3$};
	\draw (2.5,2.5) circle (0.4 cm) node {$0$};
	\draw (2.5,3.5) circle (0.4 cm) node {$4$};
	\draw (2.5,4.5) circle (0.4 cm) node {$6$};
	\draw (4.5,5.5) circle (0.4 cm) node {$0$};
	\draw (4.5,6.5) circle (0.4 cm) node {$2$};
	\draw (4.5,7.5) circle (0.4 cm) node {$6$};
	
	\node at (1.5,3.5) {$\ast$};
	\node at (3.5,6.5) {$\ast$};
	
	\end{tikzpicture}
	\caption{Example of an element in $\PLD(2,6)^{\ast 2}$.}
	\label{fig:pldExample1}
\end{figure*}
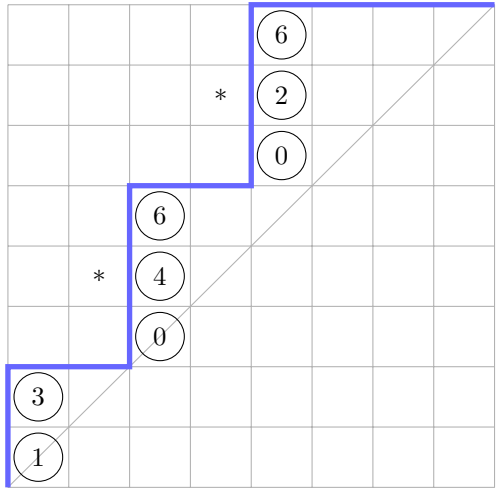

We define two statistics on this set.

\begin{definition} \label{def: area DP}
	We define the \emph{area} of a (partially labelled) decorated Dyck path $D$ as \[ \area(D) \coloneqq \sum_{i\not \in \DRise(D)} a_i(D). \]
\end{definition}

For a more visual definition, the area is the number of whole squares that lie between the path and the main diagonal, except for the ones in the rows containing a decorated rise. For example, the decorated Dyck path in Figure~\ref{fig:pldExample1} has area $7$. 

Notice that the area does not depend on the labels. 

\begin{definition} \label{def: dinv DP}
	Let $D \in \PLD(m,n)$. For $1 \leq i < j \leq n+m$, we say that the pair $(i,j)$ is an \emph{inversion} if
	\begin{itemize}
		\item either $a_i(D) = a_j(D)$ and $l_i(D) < l_j(D)$ (\emph{primary inversion}),
		\item or $a_i(D) = a_j(D) + 1$ and $l_i(D) > l_j(D)$ (\emph{secondary inversion}),
	\end{itemize}
	where $l_i(D)$ denotes the label of the vertical step in the $i$-th row.
	
	Then we define \[\dinv(D)\coloneqq \# \{ 0\leq i < j \leq n+m \mid (i,j) \; \text{is an inversion} \}.\]
\end{definition}

For example, the decorated Dyck path in Figure~\ref{fig:pldExample1} has $1$ primary inversion (the pair $(2,4)$) and $2$ secondary inversions (the pairs $(2,3)$ and $(5,6)$), so its dinv is $3$. 

Notice that the decorations on the rises do not affect the dinv.

\begin{definition}
	We define a formal series in the variables $\underline x=(x_1,x_2,\dots)$ and coefficients in $\mathbb N [q,t]$ 
	\[
	\gendelta \coloneqq \sum_{D\in \PLD(m,n)^{\ast k}}q^{\dinv(D)} t^{\area(D)}x^D.
	\]
\end{definition}

The following conjecture is stated in \cite{Haglund-Remmel-Wilson-2015}.
\begin{conjecture}[Generalized Delta]
	For $m,n,k\in \mathbb{N}$, $m\geq 0$ and $n>k\geq 0$,
	\[ \Delta_{h_{m}} \Delta'_{e_{n-k-1}} e_{n} = \gendelta . \]
\end{conjecture}

\section{Symmetric functions}\label{sec: symfun}
For all the undefined notations and the unproven identities, we refer to \cite{DAdderio-VandenWyngaerd-2017}*{Section~1}, where definitions, proofs and/or references can be found. In the next subsection we will limit ourselves to introduce some notation, while in the following one we will recall some identities that are going to be useful in the sequel. In the third subsection we prove a crucial theorem, that we are going to use in the fourth and final subsection, where we will prove the main results on symmetric functions of this work.

For more references on symmetric functions cf. also \cite{Macdonald-Book-1995}, \cite{Stanley-Book-1999} and \cite{Haglund-Book-2008}.

\subsection{Notation}

We denote by $\Lambda=\bigoplus_{n\geq 0}\Lambda^{(n)}$ the graded algebras of symmetric functions with coefficients in $\mathbb{Q}(q,t)$, and by $\<\, , \>$ the \emph{Hall scalar product} on $\Lambda$, which can be defined by saying that the Schur functions form an orthonormal basis.

The standard bases of the symmetric functions that will appear in our
calculations are the monomial $\{m_\lambda\}_{\lambda}$, complete $\{h_{\lambda}\}_{\lambda}$, elementary $\{e_{\lambda}\}_{\lambda}$, power $\{p_{\lambda}\}_{\lambda}$ and Schur $\{s_{\lambda}\}_{\lambda}$ bases.

\emph{We will use implicitly the usual convention that $e_0 = h_0 = 1$ and $e_k = h_k = 0$ for $k < 0$.}

For a partition $\mu\vdash n$, we denote by
\begin{align}
	\widetilde{H}_{\mu} \coloneqq \widetilde{H}_{\mu}[X]=\widetilde{H}_{\mu}[X;q,t]=\sum_{\lambda\vdash n}\widetilde{K}_{\lambda \mu}(q,t)s_{\lambda}
\end{align}
the \emph{(modified) Macdonald polynomials}, where
\begin{align}
	\widetilde{K}_{\lambda \mu} \coloneqq \widetilde{K}_{\lambda \mu}(q,t)=K_{\lambda \mu}(q,1/t)t^{n(\mu)}\quad \text{ with }\quad n(\mu)=\sum_{i\geq 1}\mu_i(i-1)
\end{align}
are the \emph{(modified) Kostka coefficients} (see \cite{Haglund-Book-2008}*{Chapter~2} for more details). 

The set $\{\widetilde{H}_{\mu}[X;q,t]\}_{\mu}$ is a basis of the ring of symmetric functions $\Lambda$. This is a modification of the basis introduced by Macdonald \cite{Macdonald-Book-1995}.

If we identify the partition $\mu$ with its Ferrers diagram, i.e. with the collection of cells $\{(i,j)\mid 1\leq i\leq \mu_j, 1\leq j\leq \ell(\mu)\}$, then for each cell $c\in \mu$ we refer to the \emph{arm}, \emph{leg}, \emph{co-arm} and \emph{co-leg} (denoted respectively as $a_\mu(c), l_\mu(c), a_\mu(c)', l_\mu(c)'$) as the number of cells in $\mu$ that are strictly to the right, above, to the left and below $c$ in $\mu$, respectively.

We set $M \coloneqq (1-q)(1-t)$ and we define for every partition $\mu$
\begin{align}
	B_{\mu} & \coloneqq B_{\mu}(q,t)=\sum_{c\in \mu}q^{a_{\mu}'(c)}t^{l_{\mu}'(c)} \\
	T_{\mu} & \coloneqq T_{\mu}(q,t)=\prod_{c\in \mu}q^{a_{\mu}'(c)}t^{l_{\mu}'(c)} \\
	\Pi_{\mu} & \coloneqq \Pi_{\mu}(q,t)=\prod_{c\in \mu/(1)}(1-q^{a_{\mu}'(c)}t^{l_{\mu}'(c)}) \\
	w_{\mu} & \coloneqq w_{\mu}(q,t)=\prod_{c\in \mu} (q^{a_{\mu}(c)} - t^{l_{\mu}(c) + 1}) (t^{l_{\mu}(c)} - q^{a_{\mu}(c) + 1}).
\end{align}

We will make extensive use of the \emph{plethystic notation} (cf. \cite{Haglund-Book-2008}*{Chapter~1}).

We have for example the addition formulas
\begin{align}
	\label{eq:e_h_sum_alphabets}
	e_n[X+Y]=\sum_{i=0}^ne_{n-i}[X]e_i[Y]\quad \text{ and } \quad  h_n[X+Y]=\sum_{i=0}^nh_{n-i}[X]h_i[Y].
\end{align}
We will also use the symbol $\epsilon$ for
\begin{equation}
	f[\epsilon X] = (-1)^k f[X]\qquad \text{ for } f[X]\in \Lambda^{(k)},
\end{equation}
so that, in general,
\begin{align}
	\label{eq:minusepsilon}
	f[-\epsilon X] = \omega f[X]
\end{align}
for any symmetric function $f$, where $\omega$ is the fundamental algebraic involution which sends $e_k$ to $h_k$, $s_{\lambda}$ to $s_{\lambda'}$ and $p_k$ to $(-1)^{k-1}p_k$.

Recall the \emph{Cauchy identities}
\begin{align}
	\label{eq:Cauchy_identities}
	h_n[XY] = \sum_{\lambda\vdash n} s_{\lambda}[X] s_{\lambda}[Y] \quad \text{ and } \quad h_n[XY] = \sum_{\lambda\vdash n} h_{\lambda}[X] m_{\lambda}[Y].
\end{align}

%

We define the \emph{nabla} operator on $\Lambda$ by
\begin{align}
	\nabla \widetilde{H}_{\mu} \coloneqq T_{\mu} \widetilde{H}_{\mu} \quad \text{ for all } \mu,
\end{align}
and we define the \emph{delta} operators $\Delta_f$ and $\Delta_f'$ on $\Lambda$ by
\begin{align}
	\Delta_f \widetilde{H}_{\mu} \coloneqq f[B_{\mu}(q,t)] \widetilde{H}_{\mu} \quad \text{ and } \quad 
	\Delta_f' \widetilde{H}_{\mu}  \coloneqq f[B_{\mu}(q,t)-1] \widetilde{H}_{\mu}, \quad \text{ for all } \mu.
\end{align}
Observe that on the vector space of symmetric functions homogeneous of degree $n$, denoted by $\Lambda^{(n)}$, the operator $\nabla$ equals $\Delta_{e_n}$. Moreover, for every $1\leq k\leq n$,
\begin{align}
	\label{eq:deltaprime}
	\Delta_{e_k} = \Delta_{e_k}' + \Delta_{e_{k-1}}' \quad \text{ on } \Lambda^{(n)},
\end{align}
and for any $k > n$, $\Delta_{e_k} = \Delta_{e_{k-1}}' = 0$ on $\Lambda^{(n)}$, so that $\Delta_{e_n}=\Delta_{e_{n-1}}'$ on $\Lambda^{(n)}$.

\medskip

For a given $k\geq 1$, we define the \emph{Pieri coefficients} $c_{\mu \nu}^{(k)}$ and $d_{\mu \nu}^{(k)}$ by setting
\begin{align}
	\label{eq:def_cmunu} h_{k}^\perp \widetilde{H}_{\mu}[X] & =\sum_{\nu \subset_k \mu} c_{\mu \nu}^{(k)} \widetilde{H}_{\nu}[X], \\
	\label{eq:def_dmunu} e_{k}\left[\frac{X}{M}\right] \widetilde{H}_{\nu}[X] & = \sum_{\mu \supset_k \nu} d_{\mu \nu}^{(k)} \widetilde{H}_{\mu}[X],
\end{align}
where $\nu\subset_k \mu$ means that $\nu$ is contained in $\mu$ (as Ferrers diagrams) and $\mu/\nu$ has $k$ lattice cells, and the symbol $\mu \supset_k \nu$ is analogously defined. The following identity is well-known:
\begin{align}
	\label{eq:rel_cmunu_dmunu} 
	c_{\mu \nu}^{(k)} = \frac{w_{\mu}}{w_{\nu}}d_{\mu \nu}^{(k)}.
\end{align}

The following summation formula is also well-known (e.g. cf. \cite{DAdderio-VandenWyngaerd-2017}*{Equation~1.35}):
\begin{equation} \label{eq:sumBmu}
	\sum_{\nu\subset_1\mu}c_{\mu \nu}^{(1)} = B_\mu,
\end{equation}
while the following one is proved right after Equation~(5.4) in \cite{DAdderio-VandenWyngaerd-2017}: for $\alpha\vdash n$,
\begin{equation} \label{eq:summation}
	\sum_{\nu\subset_\ell \alpha}c_{\alpha \nu}^{(\ell)}  T_\nu = e_{n-\ell}[B_\alpha].
\end{equation}

\medskip
%
%

Recall also the standard notation for $q$-analogues: for $n, k\in \mathbb{N}$, we set
\begin{align}
	[0]_q \coloneqq 0, \quad \text{ and } \quad [n]_q \coloneqq \frac{1-q^n}{1-q} = 1+q+q^2+\cdots+q^{n-1} \quad \text{ for } n \geq 1,
\end{align}
\begin{align}
	[0]_q! \coloneqq 1 \quad \text{ and }\quad [n]_q! \coloneqq [n]_q[n-1]_q \cdots [2]_q [1]_q \quad \text{ for } n \geq 1,
\end{align}
and
\begin{align}
	\qbinom{n}{k}_q  \coloneqq \frac{[n]_q!}{[k]_q![n-k]_q!} \quad \text{ for } n \geq k \geq 0, \quad \text{ while } \quad \qbinom{n}{k}_q \coloneqq 0 \quad \text{ for } n < k.
\end{align}

Recall also (cf. \cite{Stanley-Book-1999}*{Theorem~7.21.2}) that
\begin{align}
	\label{eq:h_q_binomial}
	h_k[[n]_q] = \qbinom{n+k-1}{k}_q \quad \text{ for } n \geq 1 \text{ and } k \geq 0,
\end{align}
and
\begin{align} \label{eq:e_q_binomial}
	e_k[[n]_q] = q^{\binom{k}{2}} \qbinom{n}{k}_q \quad \text{ for all } n, k \geq 0.
\end{align}


\subsection{Some useful identities}


We will use the following form of \emph{Macdonald-Koornwinder reciprocity}: for all nonempty partitions $\alpha$ and $\beta$
\begin{align}
	\label{eq:Macdonald_reciprocity}
	\frac{\widetilde{H}_{\alpha}[MB_{\beta}]}{\Pi_{\alpha}} = \frac{\widetilde{H}_{\beta}[MB_{\alpha}]}{\Pi_{\beta}}.
\end{align}
The following identity is also known as \emph{Cauchy identity}:
\begin{align}
	\label{eq:Mac_Cauchy}
	e_n \left[ \frac{XY}{M} \right] = \sum_{\mu \vdash n} \frac{ \widetilde{H}_{\mu} [X] \widetilde{H}_\mu [Y]}{w_\mu} \quad \text{ for all } n.
\end{align}

We need the following well-known proposition.
\begin{proposition} 
	For $n\in \mathbb{N}$ we have
	\begin{align}
		\label{eq:en_expansion}
		e_n[X] = e_n \left[ \frac{XM}{M} \right] = \sum_{\mu \vdash n} \frac{M B_\mu \Pi_{\mu} \widetilde{H}_\mu[X]}{w_\mu}.
	\end{align}
	Moreover, for all $k\in \mathbb{N}$ with $0\leq k\leq n$, we have
	\begin{align}
		\label{eq:e_h_expansion}
		h_k \left[ \frac{X}{M} \right] e_{n-k} \left[ \frac{X}{M} \right] = \sum_{\mu \vdash n} \frac{e_k[B_\mu] \widetilde{H}_\mu[X]}{w_\mu},
	\end{align}
	and
	\begin{align}
		\label{eq:p_expansion}
		\omega (p_n[X]) = [n]_q[n]_t\sum_{\mu \vdash n} \frac{M\Pi_\mu\widetilde{H}_\mu[X]}{w_\mu}.
	\end{align}
\end{proposition}
We will make use of \cite{Haglund-Schroeder-2004}*{Theorem~2.6}, i.e. for any $A,F\in \Lambda$ homogeneous
\begin{equation} \label{eq:HaglundThm}
	\sum_{\mu\vdash n}\Pi_\mu F[MB_\mu]d_{\mu\nu}^A=\Pi_\nu\left(\Delta_{A[MX]}F[X]\right)[MB_\nu],
\end{equation}
where $d_{\mu\nu}^A$ is the generalized Pieri coefficient defined by
\begin{equation}
	\sum_{\mu\supset \nu}d_{\mu\nu}^A\widetilde{H}_\mu=A\widetilde{H}_\nu.
\end{equation}

We will use the following theorem from \cite{DAdderio-VandenWyngaerd-2017}.
\begin{theorem}[\cite{DAdderio-VandenWyngaerd-2017}*{Theorem~3.1}]
	For $b,k\geq 1$ and $m\geq 0$, we have
	\begin{align}  \label{eq:mastereq}
		& \hspace{-1cm} \sum_{\gamma\vdash b}\frac{\widetilde{H}_\gamma[X]}{w_\gamma} h_k[(1-t)B_\gamma]e_m[B_\gamma]= \\
		\notag & =\sum_{j=0}^{m} t^{m-j}\sum_{s=0}^{k}q^{\binom{s}{2}} \qbinom{s+j}{s}_q \qbinom{k+j-1}{s+j-1}_q h_{s+j}\left[\frac{X}{1-q}\right] h_{m-j}\left[\frac{X}{M}\right] e_{b-s-m}\left[\frac{X}{M}\right].
	\end{align}
\end{theorem}

\subsection{A crucial theorem at $t=0$}

We introduce the following notation: for $i\geq 1$, $j\geq 0$ and $m\geq 0$
\begin{equation}
	A(i,j,m):=\sum_{\mu\vdash i+j}e_{m}[B_\mu]e_i[B_\mu]\frac{\widetilde{H}_\mu[X]}{w_\mu}.
\end{equation}
It is easy to see that $A(i,j,m)$ is symmetric in $q$ and $t$.

The goal of this subsection is to prove the following theorem.
\begin{theorem}
	For $i\geq 1$, $j\geq 0$ and $m\geq 0$, we have  
	\begin{align} \label{crucial_thm}
		\left. A(i,j,m) \right|_{q=0} & =  \sum_{s=0}^{i}t^{\binom{i-s}{2}} \qbinom{m}{i-s}_t \qbinom{s+m}{s}_t  \left.h_{s+m}\left[\frac{X}{1-t}\right]  e_{i+j-s-m}\left[\frac{X}{M}\right] \right|_{q=0} .
	\end{align}
	Moreover, we get an equivalent identity if we exchange $q$ and $t$ everywhere in this formula.
\end{theorem}
\begin{proof}
	The last statement follows easily from the symmetry of $A(i,j,m)$ in $q$ and $t$.
	
	So we will prove the formula \eqref{crucial_thm} with $q$ and $t$ interchanged. 
	
	Recall from \eqref{eq:e_h_expansion} that
	\begin{equation}
		h_i\left[\frac{X}{M}\right] e_j\left[\frac{X}{M}\right] = \sum_{\mu\vdash i+j}e_i[B_\mu]\frac{\widetilde{H}_\mu[X]}{w_\mu}.
	\end{equation}
	We want a formula for
	\begin{equation}
		\left. \Delta_{e_{m}}\left(h_i\left[\frac{X}{M}\right] e_j\left[\frac{X}{M}\right]\right) \right|_{t=0} = \left. \sum_{\mu\vdash i+j}e_{m}[B_\mu]e_i[B_\mu]\frac{\widetilde{H}_\mu[X]}{w_\mu} \right|_{t=0} .
	\end{equation}
	
	We start by observing that
	\begin{equation}
		\left. h_i\left[\frac{X}{M}\right] e_j\left[\frac{X}{M}\right]\right|_{t=0} = \left. h_i\left[\frac{X}{1-q}\right] e_j\left[\frac{X}{M}\right]\right|_{t=0}
	\end{equation}
	
	The following proposition is proved in \cite{Garsia-Hicks-Stout-2011}.
	\begin{proposition} [\cite{Garsia-Hicks-Stout-2011}*{Proposition~2.6}]
		For $i\geq 1$ and $j\geq 0$ we have
		\begin{equation} \label{eq:GHS_2_6}
			h_i\left[\frac{X}{1-q}\right] e_j\left[\frac{X}{M}\right]  = \sum_{\mu\vdash i+j} \frac{\widetilde{H}_\mu[X]}{w_\mu} \sum_{r=1}^i \begin{bmatrix}
				i-1\\
				r-1
			\end{bmatrix}_q q^{\binom{r}{2}+r-ir}(-1)^{i-r}h_r[(1-t)B_\mu].
		\end{equation}
	\end{proposition}
	We get
	\begin{align*} 
		\left.\Delta_{e_{m}}\left(h_i\left[\frac{X}{M}\right] e_j\left[\frac{X}{M}\right]\right)\right|_{t=0}& =\left.\Delta_{e_{m}}\left(h_i\left[\frac{X}{1-q}\right] e_j\left[\frac{X}{M}\right]\right)\right|_{t=0}\\
		\text{(using \eqref{eq:GHS_2_6})}& =  \sum_{r=1}^i \qbinom{i-1}{r-1}_q q^{\binom{r}{2}+r-ir}(-1)^{i-r} \sum_{\mu\vdash i+j} \left.\frac{\widetilde{H}_\mu[X]}{w_\mu}h_r[(1-t)B_\mu] e_{m}[B_\mu]\right|_{t=0} \\
		\text{(using \eqref{eq:mastereq})}& =  \sum_{r=1}^i \qbinom{i-1}{r-1}_q q^{\binom{r}{2}+r-ir}(-1)^{i-r} \times\\
		& \quad \times   \left.\sum_{s=0}^{r}q^{\binom{s}{2}} \qbinom{s+m}{s}_q \qbinom{r+m-1}{s+m-1}_q h_{s+m}\left[\frac{X}{1-q}\right]  e_{i+j-s-m}\left[\frac{X}{M}\right] \right|_{t=0} \\
		& =  \sum_{s=0}^{i}q^{\binom{s}{2}} \qbinom{s+m}{s}_q  \left.h_{s+m}\left[\frac{X}{1-q}\right]  e_{i+j-s-m}\left[\frac{X}{M}\right] \right|_{t=0} \times\\
		& \quad \times  \sum_{r=\max (1,s)}^i \qbinom{i-1}{r-1}_q q^{\binom{r}{2}+r-ir}(-1)^{i-r}\qbinom{r+m-1}{s+m-1}_q \\
		& =  \sum_{s=0}^{i}q^{\binom{i-s}{2}} \qbinom{m}{i-s}_q \qbinom{s+m}{s}_q  \left.h_{s+m}\left[\frac{X}{1-q}\right]  e_{i+j-s-m}\left[\frac{X}{M}\right] \right|_{t=0},
	\end{align*}
	where in the last equality we used the following lemma:
	\begin{lemma}
		For $s\geq 0$, $m\geq 0$ and $i\geq 1$, we have
		\begin{equation} \label{eq:lemma_elementar1}
			\sum_{r=1}^i q^{\binom{s}{2}}\qbinom{i-1}{r-1}_q q^{\binom{r}{2}+r-ir}(-1)^{i-r}\qbinom{r+m-1}{s+m-1}_q = q^{\binom{i-s}{2}}   \qbinom{m}{i-s}_q.
		\end{equation}
	\end{lemma}
	\begin{proof}
		This is none other than \cite{DAdderio-VandenWyngaerd-2017}*{Lemma~3.6} with $a$ replaced by $-s$ and $s$ replaced by $m+s$. Notice that, even if it is stated for $a\geq 0$, the proof in \cite{DAdderio-VandenWyngaerd-2017} actually works for $a\geq -i$.
	\end{proof}
	This completes the proof of \eqref{crucial_thm} with $q$ and $t$ interchanged.
\end{proof}

\subsection{Red and blue formulae}

The following notation will be useful: for $n\geq 1$, $m\geq 0$ and $1\leq k \leq n$, we set
\begin{equation}
	{ }^tC_{n,k}^{(m)}:= \left.\Delta_{h_{m}}\Delta_{e_{k-1}}'e_n\right|_{q=0}\quad \text{ and } \quad { }^qC_{n,k}^{(m)}:= \left.\Delta_{h_{m}}\Delta_{e_{k-1}}'e_n\right|_{t=0}.
\end{equation}
\begin{remark} \label{rem:qt_symmetry}
	Since $\Delta_{h_{m}}\Delta_{e_{k-1}}'e_n$ is symmetric in $q$ and $t$, we have
	\begin{equation}
		{ }^tC_{n,k}^{(m)}=\left. { }^qC_{n,k}^{(m)}\right|_{q=t}.
	\end{equation}
\end{remark}
The goal of this subsection is to prove the following theorem.
\begin{theorem} \label{thm:red_blue_formulas}
	For $1\leq j <n$, $1\leq k\leq n$ and $m\geq 0$, we have the \emph{blue formula}
	\begin{align} \label{blue_formula}
		h_j^\perp { }^qC_{n,k}^{(m)} & = \sum_{s=0}^{\min(j,k-1)}q^{\binom{j-s}{2}} \qbinom{m}{j-s}_q\qbinom{s+m}{s}_q \textcolor{blue}{q^s}\cdot { }^qC_{n-j,k-s}^{(s+m)} \\
		\notag	& +\sum_{s=0}^{\min(j,k-1)}q^{\binom{j-s}{2}} \qbinom{m}{j-s}_q\qbinom{s+m-1}{s-1}_q\sum_{r=0}^{s+m-1} q^r({ }^qC_{n-j,k-s}^{(r)}+ { }^qC_{n-j,k-s+1}^{(r)})\\
		\notag	& + \chi(j\geq k) q^{\binom{j-k}{2}} \qbinom{m}{j-k}_q \qbinom{k+m-1}{k-1}_q  \sum_{r=0}^{k+m-1} q^r\cdot { }^qC_{n-j,1}^{(r)},
	\end{align}
	and the \emph{red formula}
	\begin{align} \label{red_formula}
		h_j^\perp { }^qC_{n,k}^{(m)} & = \sum_{s=0}^{\min(j,k-1)}q^{\binom{j-s}{2}} \qbinom{m}{j-s}_q\qbinom{s+m}{s}_q \cdot { }^qC_{n-j,k-s}^{(s+m)} \\
		\notag	& +\sum_{s=0}^{\min(j,k-1)}q^{\binom{j-s}{2}} \qbinom{m}{j-s}_q\qbinom{s+m-1}{s-1}_q\textcolor{red}{q^{k-s}}\sum_{r=0}^{s+m-1} q^r({ }^qC_{n-j,k-s}^{(r)}+ { }^qC_{n-j,k-s+1}^{(r)}) \\
		& + \chi(j\geq k) q^{\binom{j-k}{2}} \qbinom{m}{j-k}_q \qbinom{k+m-1}{k-1}_q  \sum_{r=0}^{k+m-1} q^r\cdot { }^qC_{n-j,1}^{(r)}.
	\end{align}
	Moreover, replacing $q$ by $t$ everywhere these formulae still hold.
\end{theorem}
\begin{proof}
	The last statement follows immediately from Remark~\ref{rem:qt_symmetry}.
	
	We start with a remark.	
	\begin{remark}
		Observe that
		\begin{equation}
			{ }^tC_{n,k}^{(m)}= \left.\Delta_{h_{m}}\Delta_{e_{k-1}}'e_n\right|_{q=0}=\frac{[k]_t}{[n]_t}\left.\Delta_{h_{m}}\Delta_{e_{k}}\omega(p_n)\right|_{q=0} .
		\end{equation}
		
		The argument to prove this is the same as the one appearing in the proof of Theorem~5.1 in \cite{DAdderio-Iraci-VandenWyngaerd-Delta-Square} in the case $m=0$.
	\end{remark}

	For every $j$ such that $1\leq j<n$, using \eqref{eq:p_expansion}, we have
	
	\begin{align*}
		h_j^\perp \Delta_{h_{m}}\Delta_{e_{k}}\omega(p_n) & = [n]_q [n]_t \sum_{\mu\vdash n}   M\frac{\Pi_\mu}{w_{\mu}} h_{m}[B_\mu]e_{k}[B_{\mu}]h_j^\perp\widetilde{H}_\mu[X] \\
		\text{(using \eqref{eq:def_cmunu})}	& = [n]_q [n]_t \sum_{\mu\vdash n}   M\frac{\Pi_\mu}{w_{\mu}}  h_{m}[B_\mu]e_{k}[B_{\mu}]\sum_{\nu\subset_{j} \mu} c_{\mu\nu}^{(j)}\widetilde{H}_\nu[X]\\
		\text{(using \eqref{eq:rel_cmunu_dmunu})}& = [n]_q [n]_t \sum_{\nu\vdash n-j}M\frac{\widetilde{H}_\nu[X]}{w_{\nu}} \sum_{\mu\supset_j \nu}  \Pi_\mu   h_{m}[B_\mu]e_{k}[B_{\mu}] d_{\mu\nu}^{(j)}\\
		& = [n]_q [n]_t \sum_{\nu\vdash n-j}M\frac{\widetilde{H}_\nu[X]}{w_{\nu}} \Pi_\nu\left.(\Delta_{e_j} h_m[X/M]e_k[X/M])\right|_{X=MB_\nu}\\
		\text{(using \eqref{eq:e_h_expansion})}& = [n]_q [n]_t \sum_{\nu\vdash n-j}M\frac{\widetilde{H}_\nu[X]}{w_{\nu}} \Pi_\nu \sum_{\gamma\vdash k+m} e_j[B_\gamma]e_m[B_\gamma] \frac{\widetilde{H}_\gamma[MB_\nu]}{w_\gamma},
	\end{align*}
	where in the fourth equality we used \eqref{eq:HaglundThm} with $A[X]=e_j[X/M]$ and $F[X]= h_m[X/M]e_k[X/M]$.

	Specializing at $q=0$, we get
	\begin{align} \label{eq:intermediate_q0}
		\notag	& \hspace{-1cm} h_j^\perp \frac{[k]_t}{[n]_t}\left.\Delta_{h_{m}}\Delta_{e_{k}}\omega(p_n)\right|_{q=0}=\\
		\notag	 & = [k]_t \sum_{\nu\vdash n-j}M\frac{\widetilde{H}_\nu[X]}{w_{\nu}} \left. \Pi_\nu \sum_{\gamma\vdash k+m} e_j[B_\gamma]e_m[B_\gamma] \frac{\widetilde{H}_\gamma[MB_\nu]}{w_\gamma} \right|_{q=0} \\
		\text{(using \eqref{crucial_thm})}& =  \sum_{s=0}^{j}[k]_t t^{\binom{j-s}{2}} \qbinom{m}{j-s}_t \qbinom{s+m}{s}_t \sum_{\nu\vdash n-j}M\frac{\widetilde{H}_\nu[X]}{w_{\nu}}  \left. \Pi_\nu h_{s+m}\left[B_\nu\right] e_{k-s}\left[B_\nu\right] \right|_{q=0}.
	\end{align}
	
	In the outer sum of \eqref{eq:intermediate_q0}, if $s=k$, we get
	\begin{align*}
		& \hspace{-1cm} [k]_t  t^{\binom{j-k}{2}} \qbinom{m}{j-k}_t \qbinom{k+m}{k}_t \sum_{\nu\vdash n-j}M\frac{\widetilde{H}_\nu[X]}{w_{\nu}}  \left. \Pi_\nu h_{k+m}\left[B_\nu\right]  \right|_{q=0}= \\
		& =  t^{\binom{j-k}{2}} \qbinom{m}{j-k}_t \qbinom{k+m-1}{k-1}_t  \sum_{\nu\vdash n-j}M\frac{\widetilde{H}_\nu[X]}{w_{\nu}}  \left. \Pi_\nu [k+m]_th_{k+m}\left[B_\nu\right]  \right|_{q=0}\\
		& =  t^{\binom{j-k}{2}} \qbinom{m}{j-k}_t \qbinom{k+m-1}{k-1}_t  \sum_{\nu\vdash n-j}M\frac{\widetilde{H}_\nu[X]}{w_{\nu}} \Pi_\nu \sum_{r=0}^{k+m-1}\left.  t^rB_\nu  h_r[B_\nu]\right|_{q=0}\\
		\text{(using \eqref{eq:en_expansion})}& =  t^{\binom{j-k}{2}} \qbinom{m}{j-k}_t \qbinom{k+m-1}{k-1}_t  \sum_{r=0}^{k+m-1} t^r \left. \Delta_{h_{r}}e_{n-j}\right|_{q=0} \\
		& =  t^{\binom{j-k}{2}} \qbinom{m}{j-k}_t \qbinom{k+m-1}{k-1}_t  \sum_{r=0}^{k+m-1} t^r C_{n-j,1}^{(r)} ,
	\end{align*}
	where in the second equality we used the following lemma:
	\begin{lemma}
		For $a\geq 1$ and any nonempty partition $\nu$, we have
		\begin{equation} \label{eq:lemmetto}
			\left.[a]_t h_a[B_\nu]\right|_{q=0}= \sum_{r=0}^{a-1}\left.  t^rB_\nu  h_r[B_\nu]\right|_{q=0} .
		\end{equation}
	\end{lemma}
	\begin{proof}[Proof of the Lemma]
		Using \eqref{eq:h_q_binomial}, we have
		\begin{align} 
			\left. [a]_t h_a[B_\nu]\right|_{q=0} & = [a]_t\qbinom{\ell(\nu)+a-1}{a}_t\\
			\notag	& =  [\ell(\nu)+a-1]_t\qbinom{\ell(\nu)+(a-1)-1}{a-1}_t\\
			\notag	& =  t^{a-1}[\ell(\nu)]_t\qbinom{\ell(\nu)+(a-1)-1}{a-1}_t + [a-1]_t\qbinom{\ell(\nu)+(a-1)-1}{a-1}_t\\
			\notag	\text{(by induction)}& =\sum_{r=0}^{a-1}t^r[\ell(\nu)]_t\qbinom{\ell(\nu)+r-1}{r}_t\\
			\notag \text{(using \eqref{eq:h_q_binomial})}	& = \sum_{r=0}^{a-1}\left.  t^rB_\nu  h_r[B_\nu]\right|_{q=0} .
		\end{align}
	\end{proof}
	On the other hand, again in the outer sum of \eqref{eq:intermediate_q0}, if $s<k$, then the corresponding internal sum gives
	\begin{align*}
		\sum_{\nu\vdash n-j}M\frac{\widetilde{H}_\nu[X]}{w_{\nu}}  \left. \Pi_\nu h_{s+m}\left[B_\nu\right] e_{k-s}\left[B_\nu\right] \right|_{q=0} & =\frac{1}{[k-s]_t}\left.\Delta_{h_{s+m}}\Delta_{e_{k-s}}\omega(p_{n-j})\right|_{q=0}	,
	\end{align*}
	so that
	\begin{align*}
		h_j^\perp { }^t C_{n,k}^{(m)} & = \sum_{s=0}^{\min(j,k-1)}t^{\binom{j-s}{2}} \qbinom{m}{j-s}_t\qbinom{s+m}{s}_t \frac{[k]_t}{[k-s]_t} { }^tC_{n-j,k-s}^{(s+m)} \\
		& + \chi(j\geq k) t^{\binom{j-k}{2}} \qbinom{m}{j-k}_t \qbinom{k+m-1}{k-1}_t  \sum_{r=0}^{k+m-1} t^r \cdot { }^tC_{n-j,1}^{(r)}.
	\end{align*}
	The first sum can be developed in two ways: either using $[k]_t=t^{s}[k-s]_t+[s]_t$ or using $[k]_t=t^{k-s}[s]_t+[k-s]_t$. Using the first formula for $[k]_t$ we get
	\begin{align*}
		h_j^\perp  { }^tC_{n,k}^{(m)} & = \sum_{s=0}^{\min(j,k-1)}t^{\binom{j-s}{2}} \qbinom{m}{j-s}_t\qbinom{s+m}{s}_t \textcolor{blue}{t^s}\cdot { }^tC_{n-j,k-s}^{(s+m)} \\
		& +\sum_{s=0}^{\min(j,k-1)}t^{\binom{j-s}{2}} \qbinom{m}{j-s}_t\qbinom{s+m}{s}_t \frac{[s]_t}{[k-s]_t}{ }^tC_{n-j,k-s}^{(s+m)}\\
		& + \chi(j\geq k) t^{\binom{j-k}{2}} \qbinom{m}{j-k}_t \qbinom{k+m-1}{k-1}_t  \sum_{r=0}^{k+m-1} t^r \cdot { }^tC_{n-j,1}^{(r)}.
	\end{align*}
	Now
	\begin{align*}
		\qbinom{s+m}{s}_t \frac{[s]_t}{[k-s]_t} { }^tC_{n-j,k-s}^{(s+m)} & =\qbinom{s+m-1}{s-1}_t   \sum_{\nu\vdash n-j}M\frac{\widetilde{H}_\nu[X]}{w_{\nu}}  \left. \Pi_\nu [s+m]_t h_{s+m}\left[B_\nu\right] e_{k-s}\left[B_\nu\right] \right|_{q=0}\\
		\text{(using \eqref{eq:lemmetto})}& =\qbinom{s+m-1}{s-1}_t  \sum_{\nu\vdash n-j}M\frac{\widetilde{H}_\nu[X]}{w_{\nu}}  \left. \Pi_\nu \sum_{r=0}^{s+m-1} t^rB_\nu  h_r[B_\nu] e_{k-s}\left[B_\nu\right] \right|_{q=0}\\
		\text{(using \eqref{eq:en_expansion})}& =\qbinom{s+m-1}{s-1}_t  \left.  \sum_{r=0}^{s+m-1} t^r \Delta_{h_r}\Delta_{e_{k-s}}e_{n-j} \right|_{q=0}\\
		\text{(using \eqref{eq:deltaprime})}& =\qbinom{s+m-1}{s-1}_t  \left.  \sum_{r=0}^{s+m-1} t^r (\Delta_{h_r}\Delta_{e_{k-s-1}}'e_{n-j}+\Delta_{h_r}\Delta_{e_{k-s}}'e_{n-j} )\right|_{q=0}\\
		& =\qbinom{s+m-1}{s-1}_t \sum_{r=0}^{s+m-1} t^r ({ }^tC_{n-j,k-s}^{(r)}+{ }^tC_{n-j,k-s+1}^{(r)}),
	\end{align*}
	which gives precisely the blue formula \eqref{blue_formula} with $q$ replaced by $t$, i.e.
	\begin{align}
		h_j^\perp { }^tC_{n,k}^{(m)} & = \sum_{s=0}^{\min(j,k-1)}t^{\binom{j-s}{2}} \qbinom{m}{j-s}_t\qbinom{s+m}{s}_t \textcolor{blue}{t^s}\cdot { }^tC_{n-j,k-s}^{(s+m)} \\
		\notag	& +\sum_{s=0}^{\min(j,k-1)}t^{\binom{j-s}{2}} \qbinom{m}{j-s}_t\qbinom{s+m-1}{s-1}_t\sum_{r=0}^{s+m-1} t^r({ }^tC_{n-j,k-s}^{(r)}+ { }^tC_{n-j,k-s+1}^{(r)})\\
		\notag	& + \chi(j\geq k) t^{\binom{j-k}{2}} \qbinom{m}{j-k}_t \qbinom{k+m-1}{k-1}_t  \sum_{r=0}^{k+m-1} t^r \cdot { }^tC_{n-j,1}^{(r)}.
	\end{align}
	The proof of the red formula \eqref{red_formula} with $q$ replaced by $t$ is analogous, except that it uses the identity $[k]_t=t^{k-s}[s]_t+[k-s]_t$ instead.
	
	Replacing $t$ by $q$ everywhere, and using Remark~\ref{rem:qt_symmetry}, gives the blue formula \eqref{blue_formula} and the red formula \eqref{red_formula}, completing the proof of our theorem.
\end{proof}

\subsection{A useful lemma}

The following lemma will be useful.

\begin{lemma} \label{lem:j=n}
	For $1\leq k\leq n$ and $m\geq 0$ we have
	\begin{equation}
		h_n^\perp { }^qC_{n,k}^{(m)}=q^{\binom{n-k}{2}} \qbinom{m}{n-k}_q \qbinom{m+k-1}{k-1}_q.
	\end{equation}
\end{lemma}
\begin{proof}
	In \cite{DAdderio-Iraci-VandenWyngaerd-GenDeltaSchroeder}*{Theorem~4.7} we proved the case $\<\cdot , e_{n-d}h_d\>$ of the generalized Delta conjecture. So, in particular, we have the case $\<\Delta_{h_{m}} \Delta'_{e_{k-1}} e_{n} , h_n\> = h_n^\perp \Delta_{h_{m}} \Delta'_{e_{k-1}} e_{n}$. Specializing to $t=0$ we get that $h_n^\perp { }^qC_{n,k}^{(m)}$ is the sum of $q^{\dinv(D)}$ over all $D\in \PLD(m,n)^{\ast n-k}$ of area $0$ whose dinv reading word is $1\, 2\,\cdots\,n$. Evaluating at $t=0$ the formula for $\<\Delta_{h_{m}} \Delta'_{e_{k-1}} e_{n} , h_n\>$ given in \cite{DAdderio-Iraci-VandenWyngaerd-GenDeltaSchroeder}*{Theorem~4.7} yields the desired polynomial.
\end{proof}

\section{Ordered set partitions}

The following definitions are analogous to the definitions in \cite{Wilson-Equidistribution}, except that $0$ is added to the list of elements. 

\begin{definition}
	Let $m\in \mathbb N$, and let $\alpha=(\alpha_1,\alpha_2,\dots, \alpha_l )\in \mathbb N^l$ be a weak composition of $n\in \mathbb N$.  An \emph{ordered multiset partition} of \emph{type} $(m,\alpha)$ is a ordered partition of the multiset $M \coloneqq \{0^m\}\cup \{ i^{\alpha_i} \mid 1\leq i\leq l \}$ into sets called \emph{blocks}. So even though an element might appear multiple times, it appears at most once in each block. The set of such ordered multiset partitions of type $(m,\alpha)$ with $m+k$ blocks is denoted by $\OP(m,\alpha)^k$, and we set $\OP(m,n)^k:=\cup_\alpha \OP(m,\alpha)^k$ where $\alpha$ runs over all weak compositions of $n$. 
\end{definition}

We will represent elements of $\OP(m,n)^k$ in the following way.

\begin{example}
	Let $m=2$, $n=6$ and $k=4$. To represent an element of $\OP(2,6)^2$, we separate the blocks with a $\mid$, for example \[ 10\mid 1\mid  320 \mid  21,  \] where by convention we choose to always write elements in the same block in strictly decreasing order. The type of this ordered multiset partition is $(3,2,1)$. 
	
	Another useful way to represent an element of $\OP(2,6)^2$ is to indicate that elements belong to the same block by adding a $\ast$ between them: e.g. the previous example gets written as \[1_\ast 0 \; 1\; 3_\ast2_\ast 0 \, 2_\ast 1. \]
\end{example}

\begin{definition}
	Given $\pi\in \OP(m,n)^k $, we define $\inv(\pi)$ as the number of pairs $(i,j)$ such that 
	\begin{enumerate}[(i)]
		\item $i$ appears in $\pi$ to the left of $j$;
		\item $i$ and $j$ are in distinct blocks;
		\item $j<i$;
		\item $j$ is minimal in its block. 
	\end{enumerate}
\end{definition}

\begin{example}
	The inv of the ordered multiset partition $10 \mid 1 \mid 320 \mid 21$ is 4. Indeed, reading the elements from left to right, the second $0$ creates inv with the first two $1$'s and the last $1$ creates inv with the $3$ and the first $2$.  
\end{example}

\begin{definition}
	Given $\pi \in \OP(m,n)^k$, a \emph{primary diagonal inversion} of $\pi$ is a pair $(i,j)$ such that
	\begin{enumerate}[(i)]
		\item $i$ appears in $\pi$ to the left of $j$;
		\item $i$ and $j$ are in distinct blocks;
		\item $j<i$;
		\item if $i$ is the $h$-th smallest element in its block, then so is $j$.
	\end{enumerate}  
	A \emph{secondary diagonal inversion} is a pair $(i,j)$ such that 
	\begin{enumerate}[(i)]
		\item $i$ appears in $\pi$ to the left of $j$;
		\item $i$ and $j$ are in distinct blocks;
		\item $i<j$;
		\item if $i$ is the $h$-th smallest element in its block, then $j$ is the $(h+1)$-th smallest element in its block. 
	\end{enumerate} We define $\dinv(\pi)$ to be the number of diagonal inversions of $\pi$ (primary or secondary). 
\end{definition}
\begin{example}
	The dinv of the ordered multiset partition $10 \mid 1 \mid 320 \mid 21$ is 7. Indeed, reading the elements from left to right, the second $0$ creates primary dinv with the second $1$, the two $2$ create secondary dinv with the minimal elements of the blocks to their left ($2+3$), and the $3$ creates secondary dinv with the first $1$.
\end{example}

\begin{definition}
	Let $\sigma=\sigma_1\cdots \sigma_n $ be a word of integers. A \emph{descent} of $\sigma$ is an $i\in \{1,2,\dots, n-1\}$ such that $\sigma_{i}>\sigma_{i+1}$. The set of descents of $\sigma$ is denoted by $\text{Des}(\sigma)$.
\end{definition}

\begin{definition}
	Given $\pi \in \OP(m,n)^k$, consider the word $\sigma=\sigma_1\sigma_2\cdots\sigma_{m+n}$ obtained from $\pi$ by writing the blocks one after the other, from left to right, with the elements of each block in decreasing order. We define another word $w$, recursively. Set $w_0=0$ and $w_i=w_{i-1}+\chi(\sigma_i\text{\, is minimal in its block})$, where $\chi$ is the function defined as $\chi(\mathcal{P})=1$ if the statement $\mathcal{P}$ is true, and $\chi(\mathcal{P})=0$ otherwise. We define the \emph{major index} of $\pi$ as \[\maj(\pi) \coloneqq \sum_{i \in \text{Des}(\sigma)} w_i. \]  
\end{definition}
\begin{example}
	The maj of the ordered multiset partition $10 \mid 1 \mid 320 \mid 21$ is 8. Indeed, in this case $\sigma=1\, 0\, 1\, 3\, 2\, 0\, 2\, 1$, $\text{Des}(\sigma)=\{1,4,5,7\}$, and the corresponding $w$ word is $0\, 0\, 1\, 2\, 2\, 2\, 3\, 3\, 4$.
\end{example}

\begin{definition}
	Let $\pi\in \OP(m,n)^k$ and $(m,\alpha)$ its type, where $\alpha=(\alpha_0,..., \alpha_l)$. We define the monomial \[x^\pi \coloneqq x_1^{\alpha_1} \cdots x_l^{\alpha_l}.\]  
\end{definition}

\begin{definition}
	We define two subsets of $\OP(m,n)^k$, denoted  $\OP^L(m,n)^k$ and $\OP^R(m,n)^k $ that consist of the ordered multiset partitions that do not contain a zero in the leftmost and rightmost block, respectively. 
\end{definition}

\begin{proposition} \label{prop:xi_map_OPR_dinv}
	There exists a bijection \[\xi \colon S \subseteq \PLD(m,n)^{\ast n-k} \rightarrow \OP^R(m,n)^k  \] where $S$ is the subset of all partially labelled decorated Dyck paths of area $0$ that touch the diagonal $m+k+1$ times, and it preserves the dinv statistic.
\end{proposition}

\begin{proof}
	A partially labelled decorated Dyck path has area $0$ if and only if all its rises are decorated and the only vertical steps that are not rises start from the diagonal. This implies that elements of area $0$ in  $\PLD(m,n)^{\ast n-k}$ touch the diagonal $m+k+1$ times and a stretch of vertical steps is always followed by the same number of horizontal steps. Thus, taking the ``blocks'' of labels that label the different vertical stretches, from right to left, we obtain an ordered multiset partition that completely determines the path. Notice that the rightmost block does not contain a $0$ since the first label of a partially labelled Dyck path is never $0$. For example the path in Figure~\ref{fig: pld area 0} corresponds under $\xi$ to the ordered multiset partition $0 \mid 50 \mid 410 \mid 32$ in $\OP^R(3,5)^1$.
	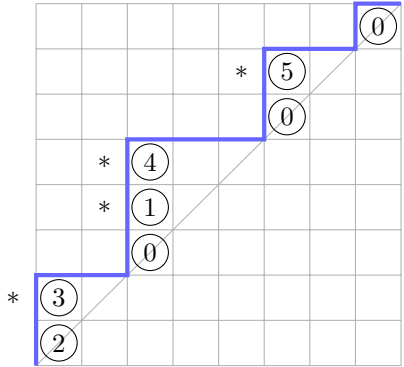
\begin{figure}[h]
		\centering
		\begin{tikzpicture}[scale=.6]
		\draw[step=1.0, gray!60, thin] (0,0) grid (8,8);
		
		\draw[gray!60, thin] (0,0) -- (8,8);
		
		\draw[blue!60, line width=1.6pt] (0,0) -- (0,1) -- (0,2) -- (1,2) -- (2,2) -- (2,3) -- (2,4) -- (2,5) -- (3,5) -- (4,5) -- (5,5) -- (5,6) -- (5,7) -- (6,7) -- (7,7) -- (7,8) -- (8,8);
		
		\draw (0.5,0.5) circle (0.4 cm) node {$2$};
		\draw (0.5,1.5) circle (0.4 cm) node {$3$};
		\draw (2.5,2.5) circle (0.4 cm) node {$0$};
		\draw (2.5,3.5) circle (0.4 cm) node {$1$};
		\draw (2.5,4.5) circle (0.4 cm) node {$4$};
		\draw (5.5,5.5) circle (0.4 cm) node {$0$};
		\draw (5.5,6.5) circle (0.4 cm) node {$5$};
		\draw (7.5,7.5) circle (0.4 cm) node {$0$};
		
		\node at (-0.5,1.5) {$\ast$};
		\node at (1.5,3.5) {$\ast$};
		\node at (1.5,4.5) {$\ast$};
		\node at (4.5,6.5) {$\ast$};
		\end{tikzpicture}
		\caption{A partially labelled decorated Dyck path of area $0$.} \label{fig: pld area 0}
	\end{figure}
	
	It is immediate to check that diagonal inversions are mapped to diagonal inversions, hence the map preserves the dinv statistic.
\end{proof}

\begin{proposition} \label{prop:xi_map_OPR_area}
	There exists a bijection \[ \eta \colon T \subseteq \PLD(m,n)^{\ast n-k} \rightarrow \OP^R(m,n)^k  \] where $T$ is the set of all partially labelled decorated Dyck paths of dinv $0$, and it maps area to maj.
\end{proposition}

\begin{proof}
	First we write the labels in the order in which they appear going top to bottom in the Dyck path, and then we declare that two consecutive letters belong to the same block if and only if they form a decorated rise in the path. In this way we get an ordered set partition with $m+k$ blocks by construction. For example the path in Figure~\ref{fig: pld dinv 0} corresponds under $\eta$ to the ordered multiset partition $310 \mid 60 \mid 5 \mid 42$ in $\OP^R(2,6)^2$.
	
	Now we compute the area in the following way: for each (not necessarily decorated) rise (which corresponds to a descent in the word $w$ of the image), we count the number of steps that are not decorated rises and that lie strictly above it (which corresponds to the number of blocks to the left of the corresponding letter in the image). If we sum all these values, then we get the area, since each step that is not a decorated rise is counted a number of times equal to its height (i.e. the vertical steps below it) minus the number of ascents (i.e. the horizontal steps below it) in the reverse pmaj reading word (i.e. the sequence of the labels going top to bottom), which is exactly the number of squares in its row. But the sum is exactly the one given in the definition of maj, so the two statistics match.
	
	\begin{figure}
		\centering
		\begin{tikzpicture}[scale=.6]
		\draw[step=1.0, gray!60, thin] (0,0) grid (8,8);
		
		\draw[gray!60, thin] (0,0) -- (8,8);
		
		\draw[blue!60, line width=1.6pt] (0,0) -- (0,1) -- (0,2) -- (0,3) -- (1,3) -- (1,4) -- (1,5) -- (2,5) -- (2,6) -- (2,7) -- (2,8) -- (3,8) -- (4,8) -- (5,8) -- (6,8) -- (7,8) -- (8,8);
		
		\draw (0.5,0.5) circle (0.4 cm) node {$2$};
		\draw (0.5,1.5) circle (0.4 cm) node {$4$};
		\draw (0.5,2.5) circle (0.4 cm) node {$5$};
		\draw (1.5,3.5) circle (0.4 cm) node {$0$};
		\draw (1.5,4.5) circle (0.4 cm) node {$6$};
		\draw (2.5,5.5) circle (0.4 cm) node {$0$};
		\draw (2.5,6.5) circle (0.4 cm) node {$1$};
		\draw (2.5,7.5) circle (0.4 cm) node {$3$};
		
		\node at (-0.5,1.5) {$\ast$};
		\node at (0.5,4.5) {$\ast$};
		\node at (1.5,6.5) {$\ast$};
		\node at (1.5,7.5) {$\ast$};		
		\end{tikzpicture}
		\caption{A partially labelled decorated Dyck path with dinv $0$.}\label{fig: pld dinv 0}
	\end{figure}
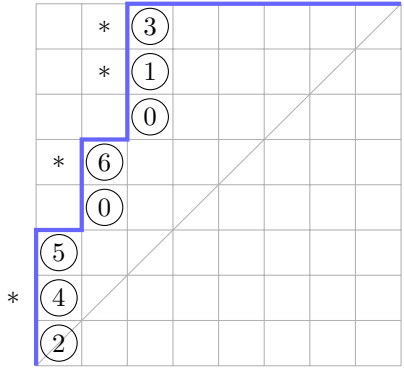
\end{proof}

\subsection{Standardization}
\begin{definition}
	Given a partially labelled Dyck path $D$, we define its \emph{dinv reading word} to be the word obtained by reading its positive labels along the diagonals, from bottom to top, from left to right. While we define its \emph{pmaj reading word} to be the word obtained by reading its positive labels along the rows, from bottom to top.
\end{definition}
\begin{example}
	For example, consider the partially labelled Dyck path of Figure~\ref{fig: pld area 0}: its dinv reading word is $2\, 3\, 1\, 5\, 4$, while its pmaj reading word is $2\, 3\, 1\, 4\, 5$.
\end{example}

\begin{definition}
	We define the set $\SOP(m,n)^k$ of \emph{standardized ordered multiset partitions} to be the set $\OP(m,(1,1,\dots,1))^k\subset \OP(m,n)^k$ of all ordered multiset partitions of type $(m,(1,1,\dots,1))$ with $m+k$ blocks. We define $\SOP^L(m,n)^k$ and $\SOP^R(m,n)^k$ as the subsets of $\SOP(m,n)^k$ of the partitions that do not contain a $0$ in their leftmost or rightmost block, respectively.
\end{definition}

Given an ordered multiset partition $\pi$, we define its \emph{dinv (resp. maj) reading word} as the dinv (resp. pmaj) reading word of the corresponding partially labelled Dyck path $\xi^{-1}(\pi)$ (possibly allowing a blank in the bottom left corner), where $\xi$ is the obvious extension of the map in Proposition~\ref{prop:xi_map_OPR_dinv}. We define its \emph{inv reading word} as its dinv reading word.

\begin{definition}
	Given an ordered multiset partition $\pi \in \OP(m,n)$ we can define the \emph{standardization} with respect to a statistic (it being inv, dinv, or maj) as the unique standard multiset partition obtained from $\pi$ by replacing its $1$'s by $1,2, \dots \alpha_1$, its $2$'s by $\alpha_1+1, \dots, \alpha_1+\alpha_2$ and so on, in such a way that the reverse reading word (with respect to the relevant statistic) of the standardization is an $\alpha$-shuffle (i.e. it is in $1, \dots, \alpha_1 \shuffle \alpha_1+1, \dots, \alpha_1+\alpha_2 \shuffle \dots$).
\end{definition}
\begin{example}
	Consider the ordered multiset partition $10\mid 2\mid  410\mid 21\in \OP^R(2,(3,2,0,1))^2$. Its standardization with respect to both dinv and inv is $10\mid 5\mid  620\mid 43$, whose reverse dinv (or inv) reading word is $6\, 1\, 2\, 4\, 5\, 3 \in 1,2,3 \shuffle 4,5 \shuffle 6$, while its standardization with respect to maj is $10 \mid 4 \mid 620 \mid 53$, whose reverse pmaj reading word is $1\, 4\, 6\, 2\, 5\, 3 \in 1,2,3 \shuffle 4,5 \shuffle 6$.
\end{example}
The verification of the following proposition is straightforward, hence left to the reader.
\begin{proposition}
	The standardization with respect to one of the statistics dinv, inv or maj preserves that statistic.
\end{proposition}

%
%

Using the well-known theory of shuffles (see e.g. \cite{Haglund-Book-2008}*{Chapter~6}), and the previous proposition, it is clear that \[ \sum_{\pi \in \OP^L(m,n)^{k}} q^{\dinv(\pi)} x^\pi = \sum_{\pi \in \SOP^L(m,n)^{k}} q^{\dinv(\pi)} Q_{\ides(\sigma(\pi))}, \]
where $\sigma(\pi)$ denotes the dinv reading word of $\pi$ and $\ides(\sigma)$ is the descent set of $\sigma^{-1}$. 

Similarly 
\begin{align*}
	\sum_{\pi \in \OP^L(m,n)^{k}} q^{\inv(\pi)} x^\pi  & = \sum_{\pi \in \SOP^L(m,n)^{k}} q^{\inv(\pi)} Q_{\ides(\sigma(\pi))} \, ,\\
	\sum_{\pi \in \OP^R(m,n)^{k}} q^{\inv(\pi)} x^\pi  & = \sum_{\pi \in \SOP^R(m,n)^{k}} q^{\inv(\pi)} Q_{\ides(\sigma(\pi))} \, ,\\
	\sum_{\pi \in \OP^R(m,n)^{k}} q^{\maj(\pi)} x^\pi  & = \sum_{\pi \in \SOP^R(m,n)^{k}} q^{\maj(\pi)} Q_{\ides(\widetilde{\sigma}(\pi))} \, ,\\
\end{align*}
where $\widetilde{\sigma}(\pi)$ denotes the maj reading word of $\pi$.

\subsection{Insertion maps}

In \cite{Wilson-Equidistribution}, Wilson proved that, given a composition $\alpha$, the three statistics defined above are equidistributed over the set of ordered multiset partitions of type $\alpha$.  He did so by constructing bijections using Carlitz's insertion method. We will describe these maps, but will not argue why they are bijective: the interested reader can find the argument in \cite{Wilson-Equidistribution}. 

\begin{theorem} \label{thm:Carlitz_bijections}
	There exist bijections
	\begin{align*}
		\phi: \OP(m,n)^k \rightarrow \OP(m,n)^k \\
		\theta: \OP(m,n)^k \rightarrow \OP(m,n)^k
	\end{align*} such that for all $\pi\in \OP(m,n)^k$ 
	\begin{align*}
		\inv(\phi(\pi))=\dinv(\pi) \\
		\inv(\theta(\pi))= \maj(\pi).
	\end{align*}
\end{theorem}

Both these maps are constructed following the same method. Start with a path in $\pi\in \OP(m,n)^k$ of type $\alpha=(\alpha_1,\alpha_2, \dots)$.  

Remove all the nonzero letters of $\pi$ one after the other, always removing a maximal letter (in specific order) and recording 
\begin{enumerate}
	\item the amount the statistic (dinv for $\phi$ and maj for $\theta$) goes down by (this may be zero);
	\item if the number of blocks goes down or not. 
\end{enumerate}
Repeating this process of deleting a maximal letter and recording this data, we obtain, after $n$ deletions, the ordered multiset partition $0\mid 0 \mid \cdots \mid 0 $. 

Next, we build the image of $\pi$, starting from $0\mid 0 \mid \cdots \mid 0 $ and inserting $\alpha_1$ $1$'s then $\alpha_2$ $2$'s and so on. We make sure the $i$-th insertion happens in a place such that inv goes up by the same amount as it went down after the $i$-th deletion. Furthermore we make sure the $i$-th insertion changes the number of blocks if and only if the $i$-th deletion did. 

If $\alpha_l$ is the last nonzero component of $\alpha$, let $\alpha^-$ be the weak composition obtained from $\alpha$ by putting $\alpha_l$ equal to zero. 

We describe insertion methods for each statistic. That is, given a multiset $D$ of size $\alpha_l$  of pairs of data $(c,b)$ where $c$ is a nonnegative integer representing the change in the statistic and $b$ a bit indicating wether the number of blocks changes ($b=1$) or not ($b=0$); we detail how to insert $\alpha_l$ $l$'s into a $\pi\in \OP(m,n-\alpha_l)^k$ of type $\alpha^-$ such that each insertion is compatible with the $(c,b)$ data.

\begin{itemize}
	\item \textit{Insertion for inv.} 
	
	\begin{enumerate}
		\item Label the blocks from $0$ to $m+k-1$, from right to left.
		\item  Define \emph{gaps} to be the positions before the first block, between two existing blocks and after the last block. Label the gaps from $0$ to $m+k$, right to left.
	\end{enumerate}
	We pick the elements of $D$ one by one from biggest $c$ to smallest and picking $(c,0)$ before $(c,1)$. For each of the $(c,0)$, we insert $l$ into the block labelled $c$, thus not changing the number of blocks and augmenting the inv by $c$. For each of the $(c,0)$'s we create a new block containing only an $l$ in the position of the gap labelled $c$, thus augmenting the inv by $c$.
	
	\item \textit{Insertion for dinv.} 
	\begin{enumerate}
		\item Let $h$ be the size of the biggest block. Label the blocks from $0$ to $m+k-1$ starting with the blocks of size $h$, from left to right, then the blocks of size $h-1$, again from left to right, etc.  
		\item Label the gaps in the same way as for the inv insertion. 
	\end{enumerate}
	Insert the $\alpha_l$ $l$'s in the same way as for the inv, with respect to this new labelling.
	
	\item \textit{Insertion for maj.} This is where the representation of ordered multiset partitions is more natural. 
	
	Let $\sigma$ be the word of nonegative integers obtained from $\pi$ by disregarding the $\ast$'s. We say that the position between two letters of an ordered multiset partition is a descent if the index of the letter directly preceding it is a descent of $\sigma$. We always consider the last gap (i.e. the position after the last block), to be a descent as well. 
	
	Label the $m+k+1$ gaps from $0$ to $m+k$ starting with the descents, from right to left, followed by the gaps that are not descents, from left to right. 
	
	Define $D'$ from $D$ by replacing all the $(c,0)$'s by $(c+1,0)$'s and pick its elements one by one starting with the elements having a maximal first component and picking $(x,1)$ before $(x,0)$.

	\begin{enumerate}
		\item For each $(c+1,0)$, insert it into the gap numbered $c+1$, move all the $\ast$'s to the right of the inserted $l$ one descent to the left and add a $\ast$ in the rightmost descent between two letters. This creates $c$ maj and no new block.
		\item For each $(c,1)$, insert it in the gap numbered $i$ and move all the $\ast$'s to the right of the inserted $n$ one descent to the left. This creates $c$ maj and a new block. 
	\end{enumerate} 
	
\end{itemize}

One can easily deduce from these insertion methods, in which order to delete $\alpha_l$ $l$'s in a partition $\pi\in \OP(m,n)^k$ of type $\alpha$ and which data to record.

\begin{example}
	We go through an example for $\phi$. Start from an element in $\OP(3,7)^3$ of type $(2,2,3)$: \[0\mid 3\mid 20\mid 1\mid 310 \mid 32.\] First, we delete the nonzero letters starting with the $3$ whose deletion has the smallest influence on the dinv. 
	\begin{center}
		\begin{tabular}{cccc}
			Deleted letter & number of blocks changes & dinv goes down by & new partition\\
			\hline
			3 & No &1 & $0\mid 3\mid 20\mid 1\mid 10 \mid 32$ \\
			3 & Yes &4 & $0\mid 20\mid 1\mid 10 \mid 32$\\
			3 & No &4 & $0\mid 20\mid 1\mid 10 \mid 2$\\
			2 & Yes &0 &$0\mid 20\mid 1\mid 10 $\\
			2 & No &2 & $0\mid 0\mid 1\mid 10  $\\
			1 & Yes &1 & $0\mid 0\mid 10 $\\
			1 & No &2 &$ 0\mid 0\mid 0 $\\
		\end{tabular} 
	\end{center}	
	Then we do the same thing but backwards and looking at the inv instead of the dinv. We start from $0\mid 0 \mid 0$. 
	
	\begin{center}
		\begin{tabular}{cccc}
			Inserted letter & number of blocks changes & inv goes up by & new partition\\
			\hline
			1 & No & 2 & $10 \mid 0  \mid 0 $\\
			1 & Yes &1 & $10 \mid 0 \mid 1 \mid 0 $\\
			2 & No &2 & $10 \mid 20 \mid 1 \mid 0 $ \\
			2 & Yes &0 & $10 \mid 20 \mid 1 \mid 0 \mid 2$\\
			3 & No &4 &  $310 \mid 20 \mid 1 \mid 0 \mid 2$\\
			3 & Yes &4 & $310\mid 3 \mid 20 \mid 1 \mid 0 \mid 2$ \\
			3 & No &1 & $310 \mid 3\mid 20 \mid 1 \mid 30 \mid 2$
	\end{tabular} \end{center}	
	So we get that \[\phi( 0\mid 3\mid 20\mid 1\mid 310 \mid 32)=310 \mid 3\mid 20 \mid 1 \mid 30 \mid 2.  \] Indeed we check that 
	\[\dinv( 0\mid 3\mid 20\mid 1\mid 310 \mid 32)=\inv(310 \mid 3\mid 20 \mid 1 \mid 30 \mid 2)=14.  \]
	
\end{example}

We look now at some restrictions of the maps $\phi$ and $\theta$. 

\begin{lemma} \label{lem:restrictions}
	We have
	\begin{enumerate}[(i)]
		\item $\phi(\OP^R(m,n)^k)=\OP^R(m,n)^k$;
		\item $\phi(\OP^L(m,n)^k)=\OP^L(m,n)^k$;
		\item $\theta(\OP^R(m,n)^k)=\OP^R(m,n)^k$.
	\end{enumerate}
\end{lemma}

\begin{proof}
	
	As $\phi$ and $\theta$ are bijections, it is enough to show the inclusions $\subseteq$ in order to prove the equalities.
	
	We prove (i) and (iii) simultaneously. 
	
	Take an element in $\OP^R(m,n)^k$, i.e. an ordered multiset partition whose rightmost block does not contain a zero. So all the letters in the rightmost block will be removed during the deletion phase of the algorithm. When the last letter of the rightmost block gets deleted, the number of blocks goes down and the dinv \emph{and} the maj goes down by $0$ units. Indeed, the rightmost gap is labelled $0$ in both the insertion methods of dinv and maj.
	
	Starting from $0\mid 0\mid \cdots \mid 0$ and inserting the nonzero letters, a necessary and sufficient condition to obtain a path whose rightmost block does not contain a zero, is that at some point we must have a block-creating insertion in the rightmost gap. Since in the inv insertion, the rightmost gap is labelled $0$, this means that we must have recorded a block-creating deletion with $0$ change of the statistic. By the argument above, this is always the case when starting from an element in $\OP^R(m,n)^k$. 
	
	The argument for (ii) is very similar. Starting from an element in $\OP^L(m,n)^k$, we will delete all the letters in the leftmost block. Let $h+1$ be the number of blocks right before we delete the last letter, call it $i$, in this block. By deleting $i$, the number of blocks goes down by $1$ and the dinv goes down by the largest value possible, $h$, or the number of blocks in the resulting partition. We use the fact that, in our algorithm, the number of blocks right before deleting $i$ is always equal to the number of blocks directly after inserting $i$. This implies that when we insert $i$, the fact that this must create a block and $h$ inv, forces us to create a new block containing only $i$ in the gap to the left of all the $h$ existing blocks. Thus, the obtained partition  will not contain a zero in its leftmost block.    
\end{proof}

\section{Reductions}

Combining Theorem~\ref{thm:Carlitz_bijections} with Lemma~\ref{lem:restrictions}, we get immediately the following corollary.
\begin{corollary} \label{cor:equivalences}
	We have
	\begin{equation}
		\sum_{\pi \in \OP^R(m,n)^{k}} q^{\inv(\pi)} x^\pi  = \sum_{\pi \in \OP^R(m,n)^{k}} q^{\dinv(\pi)} x^\pi  =
		\sum_{\pi \in \OP^R(m,n)^{k}} q^{\maj(\pi)} x^\pi .
	\end{equation}
\end{corollary}

Combining the previous corollary with Proposition~\ref{prop:xi_map_OPR_dinv}, Proposition~\ref{prop:xi_map_OPR_area} and Remark~\ref{rem:qt_symmetry}, we get immediately the following equivalence.

\begin{theorem}
	For $m,n,k\in \mathbb{N}$, $m\geq 0$ and $n>k\geq 0$,
	\begin{equation} \label{eq:key} 
		\left.\Delta_{h_{m}} \Delta'_{e_{n-k-1}} e_{n}\right|_{t=0} = \PLD_{\underline{x},q,0}(m,n)^{\ast k}
	\end{equation}
	if and only if
	\[ \left.\Delta_{h_{m}} \Delta'_{e_{n-k-1}} e_{n}\right|_{q=0} = \PLD_{\underline{x},0,t}(m,n)^{\ast k}.\]
\end{theorem}

So, in order to prove both cases $q=0$ and $t=0$ of the generalized Delta conjecture, it will be enough to prove \eqref{eq:key}: this is what we will do in the rest of this article.

\section{Combinatorial recursions}
Let  \[ \OPi^L_{\underline{x},q}(m,n)^{k} \coloneqq \sum_{\pi \in \OP^L(m,n)^{k}} q^{\inv(\pi)} x^\pi = \sum_{\pi \in \SOP^L(m,n)^{k}} q^{\inv(\pi)} Q_{\ides(\sigma(\pi))} \] be the $q$-enumerator of $\OP^L(m,n)^{k}$ for the $\inv$ statistic. 

\begin{remark}
	Extending in the obvious way the map $\xi^{-1}:\OP^R(m,n)^k\to \PLD(m,n)^{\ast n-k}$ from Proposition~\ref{prop:xi_map_OPR_dinv} to a map $\widetilde{\xi}^{-1}$ from the all set $\OP(m,n)^k$, it is easy to see that $\OPi^L_{\underline{x},q}(m,n)^{k}=\sum_{P\in S_0}q^{\dinv(P)}x^P$, where $S_0$ is the set of Dyck paths of size $m+n$ of area $0$, labelled with nonnegative integers so that the labels are strictly increasing along the columns, with the restriction that the rightmost column does not contain a $0$ (for the partially labelled Dyck paths the restriction was on the leftmost column), with $m$ labels equal to $0$ and $m+k$ columns; for these objects the corresponding monomial and the dinv are defined analogously as for the partially labelled Dyck paths. 
	
	By a standard argument using LLT polynomials (see for example the argument sketched in \cite{DAdderio-Iraci-VandenWyngaerd-Delta-Square}*{Remark~3.13}), it is easy to see that our sum over $S_0$ is indeed a symmetric function, therefore $\OPi^L_{\underline{x},q}(m,n)^{k}$ is a symmetric function as well.
\end{remark}

\begin{lemma} \label{lem:opl_inv}
	For $1\leq j <n$, $1\leq k\leq n$ and $m\geq 0$
	\begin{align*}
		h_j^\perp \OPi^L_{\underline{x},q}(m,n)^{k} & = \sum_{s=0}^{k} q^{\binom{j-s}{2}} \qbinom{m}{j-s}_q \qbinom{m+s}{s}_q \OPi^L_{\underline{x},q}(m+s,n-j)^{k-s} \\
		& + \sum_{s=0}^{k} q^{\binom{j-s}{2}} \qbinom{m}{j-s}_q q^m \qbinom{m+s-1}{s-1}_q \\
		& \times \sum_{r=1}^{m+s} \left(\OPi^L_{\underline{x},q}(m+s-r,n-j)^{k-s} + \OPi^L_{\underline{x},q}(m+s-r,n-j)^{k-s+1}\right)
	\end{align*} 	
	and \[ \OPi^L_{\underline{x},q}(m,0)^{k} = \delta_{k,0} \delta_{m,0}. \] 
\end{lemma}

\begin{proof}
	We will show this combinatorial recursion for the (finite) set $\SOP^L(m,n)^{k}$. 
	
	Let us start with the second identity. The set $\SOP^L(m,0)^{k}$ consists of the standard ordered set partitions with $m$ $0$'s, $m+k$ blocks, and no non-zero elements, with no $0$ in the leftmost block. If $m > 0$ then there must be a $0$ in the leftmost block, so the set is empty; if $k > 0$ then there must be a block containing a non-zero element, but there are none. It follows that the set is non-empty (and consists of the empty partition alone) if and only if $m = k = 0$. Recalling that $\inv(\varnothing) = 0$, we get the desired initial conditions.
	
	When applying $h_j^\perp$ we should only consider set partitions $\pi$ such that $1,2,\dots,j$ appear in $\sigma(\pi)$ in decreasing order, which means that no pair of these labels form an inversion in $\pi$ (as $h_j^\perp$ acts as $0$ on the complement). Let us call \emph{small cars} these $j$ labels.
	
	Let $s$ be the number of small cars that are the minimum of their block (equivalently, whose block does not contain a $0$). The recursive step consists of removing the small cars if their block contains a $0$, and turning them into $0$'s otherwise; next we standardize the resulting partition. If in this way we get a partition that contains a $0$ in its leftmost block, then we delete $0$'s from the leftmost blocks until it does not anymore. See Example~\ref{ex:blue_recursion}.
	
	First we remove the $j-s$ small cars from blocks that also contain $0$'s. The contribution of each of these labels is equal to the number of $0$'s in blocks to its right, and since their block also contains a $0$, these contributions must be different numbers ranging from $0$ to $m-1$, which are $q$-counted by $q^{\binom{j-s}{2}} \qbinom{m}{j-s}_q$.
	
	Now we have to distinguish two cases. If the leftmost block does not contain a small car, then the contributions of the small cars that are the minimum of their block depend only on their interlacing with the $0$'s, and these are taken into account by the factor $\qbinom{m+s}{s}_q$. If the leftmost block does contain a small car, then its contribution must be exactly $q^m$; the contribution of the remaining $s-1$ small cars is taken into account by the factor $\qbinom{m+s-1}{s-1}_q$.
	
	After the standardization, in the former case we are left with a partition in $\OP^L(m+s,n-j)^{k-s}$ and we are done. In the latter case, say that the resulting partition starts with $r$ consecutive $0$ singletons followed by a block whose minimum is non-zero, or $r-1$ $0$ singletons followed by a non-singleton block whose minimum is $0$ (in both cases $r \geq 1$, because it had a small car in its leftmost block). Then we remove these $0$'s, and we are left with a partition in $\OP^L(m+s-r,n-j)^{k-s}$ or $\OP^L(m+s-r,n-j)^{k-s+1}$ depending whether the rightmost of these $0$'s is a singleton or not (if it is, the partition loses $r$ blocks; otherwise, it loses $r-1$).
	
	This completes the proof.
\end{proof}

\begin{example} \label{ex:blue_recursion}
	Let $\pi = 2 \mid 0 \mid 9 7 3 \mid 1 0 \mid 8 6 \mid 5 0 \mid 4$ and $j=4$. First we delete all the small cars that lie in the same block of some $0$, getting $2 \mid 0 \mid 9 7 3 \mid 0 \mid 8 6 \mid 5 0 \mid 4$ ($\inv$ decreases by $1$). Then we turn the other small cars to $0$'s, getting $0 \mid 0 \mid 9 7 0 \mid 0 \mid 8 6 \mid 5 0 \mid 0$ ($\inv$ decreases by $5$). Next we delete the starting $0$'s, getting $9 7 \mid 0 \mid 8 6 \mid 5 0 \mid 0$ ($\inv$ doesn't change). Finally we standardize, getting $5 3 \mid 0 \mid 4 2 \mid 1 0 \mid 0$ ($\inv$ doesn't change). This completes the recursive step.
\end{example}



Let  \[ \OPd^R_{\underline{x},q}(m,n)^{k} \coloneqq \sum_{\pi \in \OP^R(m,n)^{k}} q^{\dinv(\pi)} x^\pi =  \sum_{\pi \in \SOP^R(m,n)^{k}} q^{\dinv(\pi)} Q_{\ides(\sigma(\pi))} \] be the $q$-enumerator for the $\dinv$ statistic. 
\begin{remark}
	Using Proposition~\ref{prop:xi_map_OPR_dinv}, we know that 
	\[\OPd^R_{\underline{x},q}(m,n)^{k} = \left. \PLD_{\underline{x},q,t}(m,n)^{\ast n-k} \right|_{t=0}.\]
	As it is known that $\PLD_{\underline{x},q,t}(m,n)^{\ast n-k}$ is a symmetric function (see for example a similar argument sketched in \cite{DAdderio-Iraci-VandenWyngaerd-Delta-Square}*{Remark~3.13}), this shows that $\OPd^R_{\underline{x},q}(m,n)^{k}$ is a symmetric function as well.
\end{remark}

The following lemma shows that $\OPd^R_{\underline{x},q}(m,n)^{k}$ satisfies the red formula \eqref{red_formula} as ${ }^qC_{n,k}^{(m)}$ does in Theorem~\ref{thm:red_blue_formulas}.

\begin{lemma} \label{lem: recursion opr_dinv}
	For $1\leq j <n$, $1\leq k\leq n$ and $m\geq 0$
	\begin{align*}
		h_j^\perp \OPd^R_{\underline{x},q}(m,n)^{k} & = \sum_{s=0}^{k} q^{\binom{j-s}{2}} \qbinom{m}{j-s}_q \qbinom{m+s}{s}_q q^s \OPd^R_{\underline{x},q}(m+s,n-j)^{k-s} \\
		& + \sum_{s=0}^{k} q^{\binom{j-s}{2}} \qbinom{m}{j-s}_q \qbinom{m+s-1}{s-1}_q \\
		& \times \sum_{r=1}^{m+s} \left(\OPi^L_{\underline{x},q}(m+s-r,n-j)^{k-s} + \OPi^L_{\underline{x},q}(m+s-r,n-j)^{k-s+1}\right)		
	\end{align*}
	and \[ \OPd^R_{\underline{x},q}(m,0)^{k} = \delta_{k,0} \delta_{m,0}. \]
\end{lemma}

\begin{proof}	
	We will show the combinatorial recursion for the (finite) set $\SOP^R(m,n)^{k}$. Let us start with the second identity. The set $\SOP^R(m,0)^{k}$ consists of the ordered set partitions with $m$ $0$'s, $m+k$ blocks, and no non-zero elements, with no $0$ in the rightmost block. As before, the set is non-empty (and consists of the empty partition alone) if and only if $m = k = 0$. Recalling that $\dinv(\varnothing) = 0$, we get the desired initial conditions.
	
	When applying $h_j^\perp$ we should only consider set partitions $\pi$ such that $1,2,\dots,j$ appear in $\sigma(\pi)$ in decreasing order, which means that no pair of these labels form a diagonal inversion in $\pi$ (as $h_j^\perp$ acts as $0$ on the complement). Let us call \emph{small cars} these $j$ labels.
	
	Let once again $s$ be the number of small cars that are the minimum of their block. The recursive step consists of removing the small cars if their block contains a $0$, turning them into $0$'s otherwise, and then standardizing the resulting partition. If in this way we get a partition that contains a $0$ in its rightmost block, then we will conclude using a different argument.
	
	First we remove the $j-s$ small cars from blocks that also contain $0$'s. The contribution of each of these small cars is equal to the number of $0$'s in blocks to its left, and since their block also contains a $0$, these contributions must be different numbers ranging from $0$ to $m-1$, which are $q$-counted by $q^{\binom{j-s}{2}} \qbinom{m}{j-s}_q$.
	
	Now we have to distinguish two cases. If the rightmost block does not contain a small car, then the contributions of the small cars that are the minimum of their block depend only on their interlacing with the $0$'s, and these are taken into account by the factor $\qbinom{m+s}{s}_q$. If the rightmost block does contain a small car, then its contribution must be exactly $q^m$; the contribution of the remaining $s-1$ small cars is taken into account by the factor $\qbinom{m+s-1}{s-1}_q$.
	
	After standardizing, in the former case we're left with a partition in $\OP^R(m+s,n-j)^{k-s}$ and we are done. In the latter case, we remove the resulting $0$ from the rightmost block, and then we move the whole block to the leftmost position, getting a partition in $\OP^L(m+s-1,n-j)^{k-s}$ if that block was not a singleton, and a partition in $\OP(m+s-1,n-j)^{k-s}$ if it were. In both cases, the $\dinv$ decreases by exactly $k-s$ (which is the number of blocks in the resulting partition whose minimum is not a $0$): if the rightmost block was not a singleton, then the primary (resp. secondary) dinv involving the non-zero elements of that block becomes secondary (resp. primary) dinv when we move the block to the left, and the only lost contribution is the one coming from the deleted $0$. Now, by Theorem~\ref{thm:Carlitz_bijections} and Lemma~\ref{lem:restrictions} we know that in $\OP^L(m+s,n-j)^{k-s}$ and in $\OP(m+s-1,n-j)^{k-s}$ the statistics $\dinv$ and $\inv$ have the same distribution, so we can use the $q$-enumerator with respect to $\inv$ instead.
	
	We also have that
	\begin{align*}
		\SOP(m+s-1,n-j)^{k-s} & = \SOP^L(m+s-1,n-j)^{k-s} \\ & \sqcup \bigsqcup_{r=2}^{m+s} \left( \SOP^L(m+s-r,n-j)^{k-s} \sqcup \SOP^L(m+s-r,n-j)^{k-s+1} \right)
	\end{align*} (up to a natural statistic-preserving bijection) by deleting all the consecutive $0$ singletons on the left, and possibly the $0$ from the leftmost non-$0$-singleton block, if there is any. Notice that this operation does not change the $\inv$.
	
	It follows that the remaining factor is \[ \sum_{r=1}^{m+s} \left(\OPi^L_{\underline{x},q}(m+s-r,n-j)^{k-s} + \OPi^L_{\underline{x},q}(m+s-r,n-j)^{k-s+1}\right) \] 
	as claimed.
	
	This completes the proof.
\end{proof}

\begin{example}
	Let $\pi = 2 \mid 0 \mid 9 7 3 \mid 1 0 \mid 8 6 \mid 5 0 \mid 4$ and $j=4$. First we delete all the small cars that lie in the same block of some $0$, getting $2 \mid 0 \mid 9 7 3 \mid 0 \mid 8 6 \mid 5 0 \mid 4$ ($\dinv$ decreases by $1$). Then we turn the other small cars to $0$'s, getting $0 \mid 0 \mid 9 7 0 \mid 0 \mid 8 6 \mid 5 0 \mid 0$ ($\dinv$ decreases by $5$). Next we delete the rightmost $0$'s, which is a singleton, getting $0 \mid 9 7 0 \mid 0 \mid 8 6 \mid 5 0 \mid 0$ ($\dinv$ decreases by $1$). Now we standardize and then apply the bijection mapping $\dinv$ to $\inv$, and in the image we delete the starting (leftmost) $0$'s, as we did in the previous recursion. This completes the recursive step. 
	
	If the rightmost block was not a singleton small car, say if $\pi = 9 2 \mid 0 \mid 7 3 \mid 1 0 \mid 8 6 \mid 5 0 \mid 4$, then while deleting it we also move the rest of the block to the leftmost position, and at that step we get $0 \mid 7 0 \mid 0 \mid 8 6 \mid 5 0 \mid 0 \mid 9$ ($\dinv$ decreases by the same amount as before), and now we know that the resulting partition does not have a $0$ in its leftmost block, so we are already done.
\end{example}

\section{Main results}
In this section we prove the main results of this work.

\begin{theorem}
	\label{thm: recursion opl_inv}
	$\OPi^L_{\underline{x},q}(m,n)^{k} = q^m \cdot{ }^qC_{n,k}^{(m)}$.
\end{theorem}
\begin{proof}
	We proceed by induction on $n$. The base case $n=0$ (or $n=1$) is straightforward to check.
	
	Let $n \geq 1$. We are going to use that if two symmetric functions $f,g\in \Lambda^{(n)}$ with $n>0$ are such that $h_j^\perp f =h_j^\perp g$ for all $1 \leq j \leq n$, then it is not hard to see that we must have $f=g$ (cf. \cite{Haglund-Rhoades-Shimozono-Advances}*{Lemma~3.6}).
	
	For $j=n$, the expression for $h_n^\perp {}^qC_{n,k}^{(m)}$ from Lemma~\ref{lem:j=n} coincides with the polynomial we get for $\OPi^L_{\underline{x},q}(m,n)^{k}$ from the recursion.
	
	For $1\leq j<n$, comparing Lemma~\ref{lem:opl_inv} and the blue formula \eqref{blue_formula} multiplied by $q^m$, and using the induction on $n$, we conclude that $h_j^\perp \OPi^L_{\underline{x},q}(m,n)^{k} = h_j^\perp q^m\cdot { }^qC_{n,k}^{(m)}$.
	
	Therefore we proved that $h_j^\perp \OPi^L_{\underline{x},q}(m,n)^{k} = h_j^\perp q^m\cdot { }^qC_{n,k}^{(m)}$ for $1\leq j\leq n$, which implies $\OPi^L_{\underline{x},q}(m,n)^{k} = q^m\cdot { }^qC_{n,k}^{(m)}$, completing the proof of the theorem.
\end{proof}

We are finally in a position to prove \eqref{eq:key}.

\begin{theorem}
	\label{thm: recursion opr_dinv}
	$\OPd^R_{\underline{x},q}(m,n)^{k} = { }^qC_{n,k}^{(m)}$ .
\end{theorem}
\begin{proof}
	The proof is similar to the proof of Theorem~\ref{thm: recursion opl_inv}.
	
	We proceed by induction on $n$. The base case $n=0$ (or $n=1$) is straightforward to check.
	
	Let $n \geq 1$. We are going to use that if two symmetric functions $f,g\in \Lambda^{(n)}$ with $n>0$ are such that $h_j^\perp f = h_j^\perp g$ for all $1\leq j \leq n$, then it is not hard to see that we must have $f=g$ (cf. \cite{Haglund-Rhoades-Shimozono-Advances}*{Lemma~3.6}).
	
	The case $j=n$ is dealt with in the proof of Lemma~\ref{lem:j=n}. 
	
	For $1\leq j<n$, consider the formula in Lemma~\ref{lem: recursion opr_dinv}.
	
	Since we know that $\OPi^L_{\underline{x},q}(m+s-r,n-j)^{k-s} = q^{m+s-r} \cdot { }^qC_{n-j,k-s}^{(m+s-r)}$ (same for the term with the extra $+1$) from Theorem~\ref{thm: recursion opl_inv}, and that ${ }^qC_{n-j,k-s}^{(m+s-r)} = \OPd^R_{\underline{x},q}(m+s-r,n-j)^{k-s}$ (same for the term with the extra $+1$) by induction on the degree, we can replace it with \[\ \sum_{r=1}^{m+s} q^{m+s-r} \left(\OPd^R_{\underline{x},q}(m+s-r,n-j)^{k-s} + \OPd^R_{\underline{x},q}(m+s-r,n-j)^{k-s+1}\right). \] 
	
	Comparing the resulting formula with the red formula \eqref{red_formula}, and using the induction on $n$, we conclude that $h_j^\perp \OPd^R_{\underline{x},q}(m,n)^{k} = h_j^\perp { }^qC_{n,k}^{(m)}$.
	
	Therefore we proved that $h_j^\perp \OPd^R_{\underline{x},q}(m,n)^{k} = h_j^\perp { }^qC_{n,k}^{(m)}$ for $1\leq j\leq n$, which implies $\OPd^R_{\underline{x},q}(m,n)^{k} = { }^qC_{n,k}^{(m)}$, completing the proof of the theorem.
\end{proof}
The following theorem is the main result of the present article.
\begin{theorem}
	For $m,n,k\in \mathbb{N}$, $m\geq 0$ and $n>k\geq 0$, we have both
	\[
	\left.\Delta_{h_{m}} \Delta'_{e_{n-k-1}} e_{n}\right|_{t=0} = \PLD_{\underline{x},q,0}(m,n)^{\ast k}
	\]
	and
	\[ \left.\Delta_{h_{m}} \Delta'_{e_{n-k-1}} e_{n}\right|_{q=0} = \PLD_{\underline{x},0,t}(m,n)^{\ast k}.\]
\end{theorem}

Combining this theorem with the fact that the generalized Delta conjecture at $q=0$ is equivalent to the generalized Delta square conjecture of \cite{DAdderio-Iraci-VandenWyngaerd-Delta-Square}*{Conjecture~3.12} at $q=0$ (see the proof of \cite{DAdderio-Iraci-VandenWyngaerd-Delta-Square}*{Theorem~5.1} for an argument), we get the following corollary, which generalizes \cite{DAdderio-Iraci-VandenWyngaerd-Delta-Square}*{Theorem~5.1}.
\begin{theorem}
	The generalized Delta square conjecture \cite{DAdderio-Iraci-VandenWyngaerd-Delta-Square}*{Conjecture~3.12} at $q=0$ holds true.
\end{theorem}

\section{Open problems}

Combining Theorem~\ref{thm: recursion opl_inv} and Theorem~\ref{thm: recursion opr_dinv} with Theorem~\ref{thm:Carlitz_bijections} and Lemma~\ref{lem:restrictions}, we proved that 
\[q^m\cdot \OPd^R_{\underline{x},q}(m,n)^{k}= \OPd^L_{\underline{x},q}(m,n)^{k}.\]

It would be interesting to find a direct (bijective?) combinatorial proof of this identity. Notice that such a proof would reduce Theorem~\ref{thm: recursion opr_dinv} to Theorem~\ref{thm: recursion opl_inv}.

\medskip

Observe also that the generalized Delta square conjecture \cite{DAdderio-Iraci-VandenWyngaerd-Delta-Square}*{Conjecture~3.12} is not symmetric in $q$ and $t$, so our results leave open the case $t=0$.


\bibliographystyle{amsalpha}
\bibliography{Biblebib}

\end{document}